\documentclass[12pt,a4paper,reqno]{amsart} 
\title{Hopf algebroids from noncommutative bundles}
\date{January 2022}
\usepackage{amsmath, amssymb, amsthm}
\usepackage{fullpage}

\usepackage{tikz-cd}
\tikzcdset{arrow style=tikz, diagrams={>=stealth}}

\author{Xiao Han, Giovanni Landi, Yang Liu}
\address[]{\textit{Xiao Han},
Queen Mary University of London.
}
\email{x.h.han@qmul.ac.uk}
\address[]{\textit{Giovanni Landi},
Universit\`a di Trieste,
Trieste, Italy
\newline \indent
and Institute for Geometry and Physics (IGAP) Trieste, Italy 
\newline \indent and INFN, Trieste, Italy}
\email{landi@units.it}
\address[]{\textit{Yang Liu},
SISSA, via Bonomea 265, 34136 Trieste, Italy}
\email{ yliu@sissa.it}

\newcommand{\brac}[1]{\ensuremath{\left( {#1} \right)}}
\newcommand{\set}[1]{\ensuremath{\left\{ {#1} \right\}}}
\newcommand{\abrac}[1]{\ensuremath{\left\langle {#1} \right\rangle}}

\newcommand{\twcoc}{\lambda}

\newcommand{\beq}{\begin{equation}}
\newcommand{\eeq}{\end{equation}}
\newcommand{\nn}{\nonumber}
\newcommand{\ot}{{\otimes}}
\newcommand{\pot}{{\dot{\otimes}}}

\newcommand{\ii}{{\,{\rm i}\,}}
\numberwithin{equation}{section}

\theoremstyle{plain}
\newtheorem{thm}{Theorem}[section]
\newtheorem{lem}[thm]{Lemma}
\newtheorem{prop}[thm]{Proposition}
 \newtheorem{cor}[thm]{Corollary}
\newtheorem{defi}[thm]{Definition}

\theoremstyle{remark}
\newtheorem{exa}[thm]{Example}
\newtheorem{rem}[thm]{Remark}
\newcommand{\cL}{\mathcal{L}}

\newcommand{\C}{\mathcal{C}}
\renewcommand{\O}{\mathcal{O}}

\newcommand{\II}{\mathbb{I}}
\newcommand{\IR}{\mathbb{R}}
\newcommand{\IC}{\mathbb{C}}
\newcommand{\IZ}{\mathbb{Z}}

\newcommand{\IT}{\mathbb{T}}

\newcommand{\wt}{\widetilde}

\DeclareMathOperator{\diag}{diag}
\DeclareMathOperator{\SU}{SU}

\DeclareMathOperator{\U}{U} 
\DeclareMathOperator{\SO}{SO}
\DeclareMathOperator{\SOt}{SO_\theta}

\DeclareMathOperator{\Aut}{Aut} 
\DeclareMathOperator{\Hom}{Hom}
\newcommand*{\id}{\textup{id}}  
\def\dd{\mathrm{d}}

\def\St{S_\theta}
\newcommand{\Sk}{S_{\theta'}}

\newcommand{\zero}[1]{{#1}{}_{\scriptscriptstyle{(0)}}}
\newcommand{\one}[1]{{#1}{}_{\scriptscriptstyle{(1)}}}
\newcommand{\two}[1]{{#1}{}_{\scriptscriptstyle{(2)}}}

\newcommand{\onet}[1]{{#1}{}_{\scriptscriptstyle{(1')}}}
\newcommand{\twot}[1]{{#1}{}_{\scriptscriptstyle{(2')}}}
\newcommand{\tuno}[1]{{#1}{}{}^{\scriptscriptstyle{<1>}}}
\newcommand{\tdue}[1]{{#1}{}{}^{\scriptscriptstyle{<2>}}}

\def\b#1{{\bf #1}}

\begin{document}
\begin{abstract}
We present two classes of examples of Hopf algebroids associated with noncommutative principal bundles. 
The first comes from deforming the principal bundle while leaving unchanged the structure Hopf algebra. 
The second is related to deforming a quantum homogeneous space; this needs a careful deformation of the structure Hopf algebra in order to preserve the compatibilities between the Hopf algebra operations. 
\end{abstract}

\maketitle
\tableofcontents
\parskip = .75 ex

\thispagestyle{empty}

\section{Introduction}
A commutative Hopf algebroid is somehow the dual of a groupoid, in the spirit of Hopf algebras versus groups. 
One is extending the scalar, similarly to the passage from Hilbert space to Hilbert module: 
the ground field $k$ gets replaced by an algebra $B$  which could be noncommutative.
The result is a bi-algebra over a noncommutative base algebra. In fact, in general not all structures survive: 
there is a notion of coproduct and counit but in general there is no antipode. 
The notions of source and target maps are still present.

An important groupoid used in gauge theory, 
is the gauge groupoid associated with a principal bundle \cite{Mac05}. 
In \cite{HL21}, as a preliminary step to study the gauge group of a noncommutative principal bundle, we considered the Ehresmann--Schauenburg bialgebroid of the noncommutative bundle which, in a sense, is the quantization of the classical gauge groupoid. For a monopole bundle over a quantum Podle\'s sphere
and a not faithfully flat Hopf--Galois extension of commutative algebras we gave a suitable invertible antipode 
so that the corresponding bialgebroids got upgraded to Hopf algebroids.

In the present paper we study two classes of examples of Hopf algebroids associated with noncommutative principal bundles. 
The first comes from deforming the principal bundle while leaving unchanged the structure Hopf algebra. 
The prototype for this is the bundle over the noncommutative four-sphere $S_\theta^{4}$ with classical $\SU(2)$ 
as structure group. 
The second class is associated to deformations of quantum homogeneous spaces. 
It is known that one needs a careful deformation of the multiplication in a Hopf algebra in order to preserve the compatibilities between the Hopf algebra structures. And this attention is needed also for deforming homogeneous spaces. Examples of the second class are the principal bundles over the noncommutative spheres $S_\theta^{2n}$ with noncommutative orthogonal group $\SO_\theta(2n, \IR)$ as structure group.  

This paper is organised as follows.  
In \S\ref{se;pr} we give a recap of algebraic preliminaries and notation, and of the relevant concepts for noncommutative principal bundles (Hopf--Galois extensions), bialgebroids and Hopf algebroids.
We devote \S\ref{se:hge} to two well know examples of Hopf--Galois extensions 
for which in \S\S ~\ref{sec:ev-q-shp} and \ref{se:alg-su2sym} we construct the corresponding Hopf algebroids; these are a $\SU(2)$-bundle over the sphere $S^4_\theta$ and $\SO_\theta(2n)$ bundles over even spheres $S^{2n}_\theta$.
In \S\ref{alg-def} we review the general scheme of deforming by the action of tori. 
This is done via $\IZ^n$-graded spaces and deforming relevant structures by means of a bi-character. 
The discussion is developed along two scenarios to cover the constructions of both 
\S\ref{subsec:ScI}, where the structure Hopf algebra is not changed, and \S\ref{subsec:ScII}
where attention is payed to a suitable deformation of the multiplication that is compatible with all Hopf algebra 
operations, in order to get new Hopf algebras with related comodule
algebras. The latter framework accommodates deformed homogeneous spaces. The noncommutative principal bundles that result from both schemes of deformation have natural Ehresmann--Schauenburg bialgebroids. In the context of the present paper the flip map will preserve the bialgebrois and will satisfy all properties for an invertible algebrois antipode. All of these last parts and the examples are worked out in \S\ref{se:had}.

\bigskip
\noindent
\subsubsection*{Acknowledgements}
We are grateful to Chiara Pagani for useful discussions. 
XH was partially supported by Marie Curie Fellowship HADG - 101027463 agreed between QMUL  and the  European Commission, and by an Assistant Professor fellowship at IMPAN. 
GL acknowledges partial support from INFN, Iniziativa Specifica GAST,
from INdAM-GNSAGA and from the INDAM-CNRS IRL-LYSM. 
YL would like to thank SISSA for the postdoctoral fellowship and excellent working environment.

\section{Preliminary results} \label{se;pr}
To be definite we work over the field $\IC$ of complex numbers. Algebras (coalgebras) are assumed to be unital and associative (counital and coassociative) with morphisms of algebras taken to be unital (of coalgebras taken counital). 
Tensor product over $\IC$ is denoted $\ot$ while the symbol $\pot$ implies also a matrix sum: for matrices $M=(m_{jk})$
and $N=(n_{kl})$ the product $M \pot N$ have components $M \pot N = (\sum_k m_{jk} \ot n_{kl})$.

\subsection{Rings and corings over an algebra}

For an algebra $B$ a {\em $B$-ring} is a triple $(A,\mu,\eta)$. Here $A$ is a $B$-bimodule with $B$-bimodule maps
$\mu:A\ot_ {B} A \to A$ and $\eta:B\to A$, satisfying associativity and unit conditions:
\begin{equation}
\mu\circ(\mu\ot _{B}  \id_A)=\mu\circ (\id_A \ot _{B} \mu), 
\quad 
\mu\circ(\eta \ot _{B} \id_A)=\id_A=\mu\circ (\id_A\ot _{B} \eta).
\end{equation}
A morphism of $B$-rings $f:(A,\mu,\eta)\to (A',\mu',\eta')$ is an
$B$-bimodule map $f:A \to A'$, such that
$f\circ \mu=\mu'\circ(f\ot_{B} f)$ and $f\circ \eta=\eta'$.

From \cite[Lemma 2.2]{Boehm} there is a bijective correspondence between $B$-rings $(A,\mu,\eta)$ and algebra automorphisms
$\eta : B \to A$. Starting with a $B$-ring $(A,\mu,\eta)$, one obtains a multiplication map $A \ot A \to A$ by composing the canonical surjection $A \ot A \to A\ot_B A$ with the map $\mu$. Conversely, starting with an algebra map $\eta : B \to A$, a $B$-bilinear associative multiplication $\mu:A\ot_ {B} A \to A$ is obtained from the universality of the coequaliser $A \ot A \to A\ot_B A$ which identifies an element $ a r \ot a'$ with $ a \ot r a'$.

Dually,  for an algebra $B$ a {\em $B$-coring} is a
triple $(C,\Delta,\varepsilon)$. Here $C$ is a $B$-bimodule with $B$-bimodule maps
$\Delta:C\to C\ot_{B} C$ and $\varepsilon: C\to B$ that satisfy coassociativity and counit conditions,
\begin{align}
(\Delta\ot _{B} \id_C)\circ \Delta = (\id_C \ot _{B} \Delta)\circ \Delta, \quad
(\varepsilon \ot _{B} \id_C)\circ \Delta = \id_C =(\id_C \ot _{B} \varepsilon)\circ \Delta.
\end{align}
A morphism of $B$-corings $f:(C,\Delta,\varepsilon)\to
(C',\Delta',\varepsilon')$ is a $B$-bimodule map $f:C \to C'$, such that
$\Delta'\circ f=(f\ot_{B} f)\circ \Delta$ and
$\varepsilon' \circ f =\varepsilon$.

Let $B$ be an algebra.
A {\em left $B$-bialgebroid} $\cL$ consists of a $(B\ot B^{op})$-ring
together with a $B$-coring structures on the same vector space $\cL$, with mutual compatibility conditions
\cite{Take77}.
From what said above, a $(B\ot B^{op})$-ring $\cL$ is the same as an algebra map $\eta : B \ot B^{op} \to \cL$.
Equivalently, one may consider the restrictions
$$
s := \eta ( \, \cdot \, \ot_B 1_B ) : B \to \cL \quad \mbox{and} \quad t := \eta ( 1_B \ot_B  \, \cdot \, ) : B^{op} \to \cL
$$
which are algebra maps with commuting ranges in $\cL$, called the \emph{source} and the \emph{target} map of the
$(B\ot B^{op})$-ring $\cL$. Thus a $(B\ot B^{op})$-ring is the same as a triple $(\cL,s,t)$ with $\cL$ an algebra and $s: B \to \cL$
and $t: B^{op} \to \cL$ both algebra maps with commuting range.

For a left $B$-bialgebroid $\cL$ the compatibility conditions are required to be the following.
\begin{itemize}
\item[i)] The bimodule structures in the $B$-coring $(\cL,\Delta,\varepsilon)$ are
related to those of the $B\ot B^{op}$-ring $(\cL,s,t)$ via
\begin{equation}\label{eq:rbgd.bimod}
b\triangleright a \triangleleft \tilde{b}:= s(b) t(\tilde{b})a , \quad \textrm{for} \,\, b, \tilde{b}\in B, \, a\in \cL.
\end{equation}

\item[ii)] Considering $\cL$ as a $B$-bimodule as in \eqref{eq:rbgd.bimod},
  the coproduct $\Delta$ corestricts to an algebra map from $\cL$ to
\begin{equation}\label{eq:Tak.prod}
\cL \times_{B} \cL := \left\{\ \sum\nolimits_j a_j\ot_{B} \tilde{a}_j\ |\ \sum\nolimits_j a_jt(b) \ot_{B} \tilde{a}_j =
\sum\nolimits_j a_j \ot_{B}  \tilde{a}_j s(b), \,\,\, \forall  \, b \in B\ \right\},
\end{equation}
where $\cL \times_{B} \cL$ is an algebra via component-wise multiplication.
\\
\item[iii)] 
The counit $\varepsilon : \cL \to B$ 
satisfies the properties,
\begin{itemize}
\item[1)] $\varepsilon(1_{\cL})=1_{B}$,   
\item[2)] $\varepsilon(s(b)a)=b\varepsilon(a) $,   
\item[3)] $\varepsilon(as(\varepsilon(\tilde{a})))=\varepsilon(a\tilde{a})=\varepsilon(at (\varepsilon(\tilde{a})))$,  \qquad 
for all $b\in B$ and $a,\tilde{a} \in \cL$.
\end{itemize}
\end{itemize}

An automorphism of the left bialgebroid $(\cL, \Delta, \varepsilon,s,t)$ over the algebra $B$ is a pair $(\Phi, \varphi)$
of algebra automorphisms, $\Phi: \cL \to \cL$, $\varphi : B \to B$ such that:
\begin{align}
\Phi\circ s & = s\circ \varphi , \qquad \Phi\circ t = t\circ \varphi ,  \label{amoeba(i)} \\
(\Phi\ot_{B} \Phi)\circ \Delta & = \Delta \circ \Phi , \qquad
\varepsilon\circ \Phi= \varphi\circ \varepsilon \label{amoeba(ii)} .
 \end{align}
 
In fact, the map $\varphi$ is uniquely determined by $\Phi$ via $\varphi = \varepsilon \circ \Phi \circ s$ 
and one can just say that $\Phi$ is a bialgebroid automorphism. Automorphisms of a bialgebroid $\cL$
form a group $\Aut(\cL)$ by map composition. A \textit{vertical} automorphism is one of the type 
$(\Phi, \varphi=\id_{B})$.

From the conditions \eqref{amoeba(i)}, $\Phi$ is a $B$-bimodule map: 
$\Phi(b \triangleright c\triangleleft \tilde{b})=b \triangleright_{\varphi}\Phi(c)\triangleleft_\varphi \tilde{b}$. 
The first condition \eqref{amoeba(ii)} is well defined once the conditions \eqref{amoeba(i)} are satisfied (the balanced tensor product is induced by $s':=s\circ \varphi $ and $t':=t\circ \varphi $). Conditions \eqref{amoeba(i)} imply $\Phi$ is a coring map, therefore $(\Phi, \varphi)$ is an isomorphism between the starting and the new bialgebroid.
  
Finally, we recall from \cite[Def. 4.1]{BS04} the conditions for a Hopf algebroid with invertible antipode. 
 Given a left bialgebroid $(\cL, \Delta, \varepsilon,s,t)$ over the algebra $B$, an invertible antipode
$S : \cL  \to  \cL $ in an algebra anti-homomorphism with inverse $S^{-1} : \cL  \to  \cL $ such that
\beq\label{hopbroid1}
S \circ t = s
\eeq
and satisfying compatibility conditions with the coproduct:
\begin{align}\label{hopbroid2}
\onet {(S^{-1}\two{h})} \ot_B \twot {(S^{-1}\two{h})}\one{h} & = S^{-1} h \ot_B 1_\cL \nn \\
\onet {(S\one{h})} \two{h} \ot_B \twot {S(\one{h})} & = 1_\cL \ot_B S h ,
\end{align}
for any $h\in \cL $. These then imply  
$
S(\one{h}) \, \two{h} = t \circ \varepsilon \circ S h. 
$

\subsection{Hopf--Galois extensions}\label{sec:hge}
We give a brief recall of Hopf--Galois extensions
as noncommutative principal bundles. These extensions are $H$-comodule algebras
$A$ with a canonically defined map $\chi: A\ot_B A\to A\ot H$
which is required to be invertible \cite{HJSc90}.

\begin{defi} \label{def:hg}
Let $H$ be a Hopf algebra and let $A$ be a $H$-comodule algebra with coaction $\delta^A$.
Consider the subalgebra $B:= A^{coH}=\big\{b\in A ~|~ \delta^A (b) = b \ot 1_H \big\} \subseteq 
A$ of coinvariant elements 
with balanced tensor product $A \ot_B A$.
The extension $B\subseteq A$ is called a $H$-\textup{Hopf--Galois extension} if the
\textit{canonical Galois map}
\begin{align*}  
\chi := (m \ot \id) \circ (\id \ot _B \delta^A ) : 
A \ot _B A \longrightarrow A \ot H ,
\quad \tilde{a} \ot_B a  &\mapsto \tilde{a} \zero{a} \ot \one{a}
\end{align*}
is an isomorphism.
\end{defi}
\begin{rem}\label{fun-rem}
For a Hopf--Galois extension $B\subseteq A$, we take the algebra $A$ to be \emph{faithfully flat} as a right $B$-module.
One possible way to state this property is that for any left $B$-module map $F : M \to N$, 
the map $F$ is injective if and only if the map $\id_A \ot_B F : A  \ot_B M \to A \ot_B N$ 
is injective; injectivity of $F$ implying the injectivity of $\id_A \ot_B F$ would state that $A$ is flat as a right $B$-module (see \cite[Chap.~13]{Wa79}.
\qed\end{rem}

Since the canonical Galois map $\chi$ is left $A$-linear, its inverse is
determined by the restriction $\tau:=\chi^{-1}_{|_{1_A \ot H}}$, named \textit{translation map},
\begin{align}
    \label{eq:tau-defn} 
\tau =\chi^{-1}_{|_{1_A \ot H}} :  H\to A\ot _B A ~ ,
\quad h \mapsto \tau(h) = \tuno{h} \ot_B \tdue{h} \, .
\end{align}
Thus by definition:
\beq
\label{p7}
\tuno{h}\zero{\tdue{h}}\ot \one{\tdue{h}} = 1_{A} \ot h  \, .
\eeq
The translation map enjoys a number of properties \cite[3.4]{HJSc90b}
that we listed here for later use. 
For any $h, k \in H$ and $a\in A$, $b\in B$:  
\begin{align}
\tuno{h} \ot_B \zero{\tdue{h}} \ot \one{\tdue{h}} &= \tuno{\one{h}} \ot_B \tdue{\one{h}} \ot \two{h} \label{p4}  \, , \\
\zero{\tuno{h}} \ot_B {\tdue{h}} \ot \one{\tuno{h}} &= \tuno{\two{h}}  \ot_B \tdue{\two{h}} \ot S(\one{h})  \, , \label{p1}
\end{align}
\begin{align}
\label{p5}
\tuno{h}\tdue{h} &= \varepsilon(h)1_{A}  \, , \\
\label{p3}
\zero{a}\tuno{\one{a}}\ot_{B}\tdue{\one{a}} &= 1_{A} \ot_{B}a \, , \\
\label{p8}
b\, \tuno{h} \ot_B \tdue{h} &= \tuno{h} \ot_B \tdue{h} \, b \,  , 
\end{align}
\begin{align}
\label{p2}
\tuno{(h k)}\ot_{B}\tdue{(h k)} &= \tuno{k}\tuno{h}\ot_{B}\tdue{h}\tdue{k}  \, , \\ 
\label{p6}
\tuno{\one{h}}\ot_{B}\tdue{\one{h}}\tuno{\two{h}}\ot_{B}\tdue{\two{h}} & 
=\tuno{h}\ot_{B}1_{A}\ot_{B}\tdue{h}  \, . 
\end{align}

\subsection{Ehresmann--Schauenburg bialgebroids} 
\label{sec:ES-bialerd}

To any Hopf--Galois extension $B=A^{co \, H}\subseteq A$ one associates a $B$-coring  
and a bialgebroid \cite{schau2} (see \cite[\S 34.13 and \S 34.14]{BW}).
These can be viewed as a quantization of the gauge groupoid that is associated to a (classical) principal fibre bundle (see  \cite{Mac05}).

The coring can be given in a few equivalent ways.
Let $B=A^{co \, H}\subseteq A$ be a 
Hopf--Galois extension with right coaction $\delta^{A} : A \to A \ot H$. This extends to 
a diagonal coaction, 
\beq\label{AAcoact}
\delta^{A\ot  A}: A\ot  A\to A\ot  A\ot  H, \quad a\ot  \tilde{a}  \mapsto
\zero{a}\ot  \zero{\tilde{a}} \ot   \one{a}\one{\tilde{a}} , \quad  \textup{for} \quad a, \tilde{a} \in A .
\eeq
Let $\tau$ be the translation map of the Hopf--Galois extension. We have the following: 		
\begin{lem}\label{lem-2vers}
The $B$-bimodule of coinvariant elements for the diagonal coaction, 
\beq  \label{ec2}
(A\ot A)^{coH} = \{a\ot  \tilde{a}\in A\ot  A \, ; \,\, \zero{a}\ot  \zero{\tilde{a}}\ot  \one{a}\one{\tilde{a}}=a\ot  \tilde{a}\ot  1_H \}
\eeq
is the same as the $B$-bimodule
\begin{equation}\label{ec1}
\C(A, H) :=\{a\ot  \tilde{a}\in A\ot  A : \,\, \zero{a}\ot  \tau(\one{a})\tilde{a}=a\ot  \tilde{a}\ot _B 1_A\}.
\end{equation}
\end{lem}
\begin{proof}
This is a direct check: using properties of the canonical map $\chi$ and of the translation map $\tau$, one shows the two inclusions. 
\end{proof}

We have then the following definition \cite{schau2} (see \cite[\S 34.13]{BW}).
\begin{defi}\label{def:ec}
Let $B=A^{co \, H}\subseteq A$ be a 
faithfully flat 
Hopf--Galois extension with translation map $\tau$.
Then the $B$-bimodule $\C(A, H)$ in \eqref{ec1} is a $B$-coring with %
coproduct,
\beq\label{copro}
\Delta(a\ot  \tilde{a}) = \zero{a}\ot  \tau(\one{a})\ot \tilde{a} 
= \zero{a} \ot \tuno{\one{a}} \ot_B \tdue{\one{a}} \ot \tilde{a} , 
\eeq
and counit,
\beq
\varepsilon(a\ot  \tilde{a}) = a\tilde{a} . \label{counit}
\eeq
\end{defi}
\noindent
Applying the map $m_A\ot \id_H$ to elements of \eqref{ec2} one gets  $a\tilde{a}\in B$.
The above $B$-coring is called the \textit{Ehresmann} or \textit{gauge coring}; we denote it $\C(A, H)$.
Also, using the well know relation between the coinvariants of a tensor product of comodules and their 
cotensor product \cite[Lemma 3.1]{HJSc90b}, the coring $\C(A, H)$ can be given as a 
cotensor product $A\, \square \, {}^H\!\!A$.

The Ehresmann coring of a Hopf--Galois extension is
in fact a bialgebroid \cite{schau2}, called the \textit{Ehresmann--Schauenburg bialgebroid} (see  \cite[34.14]{BW}).
One see that $\C(A, H) = (A\ot A)^{coH}$ is a subalgebra of $A \ot A^{op}$; indeed, given
$x \ot  \tilde{x}, \, y \ot  \tilde{y} \in (A\ot A)^{coH}$, one computes
\begin{align*}
     \delta^{A\ot  A}(x y \ot \tilde{y} \tilde{x}) 
     &=
\zero{x} \zero{y} \ot \zero{\tilde{y}} \zero{\tilde{x}} 
\ot \one{x}\one{y} \one{\tilde{y}} \one{\tilde{x}}  \\
     &=
\zero{x} y \ot \tilde{y} \zero{\tilde{x}}  \ot \one{x} \one{\tilde{x}} \\ 
&=
x y \ot \tilde{y} \tilde{x} \ot  1_H.
\end{align*}

\begin{defi}\label{def:reb}
Let $\C(A, H)$ be the coring associated with a faithfully flat Hopf--Galois extension
$B=A^{co \, H}\subseteq A$. Then $\C(A, H)$
is a (left) $B$-bialgebroid with product
$$
(x \ot  \tilde{x}) \bullet_{\C(A, H)} ({y}\ot  \tilde{y}) = x y \ot \tilde{y} \tilde{x} ,
$$
for all $x \ot  \tilde{x}, \, y \ot  \tilde{y} \in \C(A, H)$ (and unit $1_A\ot  1_A$). The target and the  source  maps are
$$
t(b)=1_A \ot  b \quad \textup{and} \quad s(b)=b\ot 1_A.
$$
\end{defi}
\noindent
We refer to \cite[34.14]{BW} for the checking that all defining properties are satisfied. When there is no risk of confusion we drop the decoration $\bullet_{\C(A, H)}$ in the product. 

\section{Two examples of Hopf--Galois extension}\label{se:hge}

We review two well know examples of Hopf--Galois extensions for which in 
\S\S ~\ref{sec:ev-q-shp} and \ref{se:alg-su2sym}
we shall explicitly construct the corresponding algebroids.

\subsection{The $\SU(2)$ principal fibration} \label{sec:HG-SU2}

Consider the sphere $\St^4$ constructed in \cite{CL01}. 
With $\theta$ a real parameter, the algebra $A(\St^4)$ of polynomial functions on the sphere $\St^4$ is 
 generated by elements  $\zeta_0=\zeta_0^*$ and $\zeta_j, \zeta_j^*$, $j=1,2$, subject to relations
\beq\label{s4t}
\zeta_\mu \zeta_\nu = \lambda_{\mu\nu} \zeta_\nu \zeta_\mu, \quad  \zeta_\mu \zeta_\nu^* = \lambda_{\nu\mu} \zeta_\nu \zeta_\mu^*,
\quad \zeta_\mu^* \zeta_\nu^* = \lambda_{\mu\nu} \zeta_\nu^* \zeta_\mu^*, \quad \mu,\nu = 0,1,2 ,
\eeq
with deformation parameters given by
\beq
\lambda_{1 2} = \bar{\lambda}_{2 1} =: \lambda=e^{2\pi \ii \theta}, 
\quad \lambda_{j 0} = \lambda_{0 j } = 1, \quad j=1,2 ,
\eeq
 and together with the spherical relation $\sum_\mu \zeta_\mu^* \zeta_\mu=1$. For $\theta=0$ one recovers 
the 
$*$-algebra of complex polynomial functions on the usual sphere $S^4$.

On the sphere $\St^4$ there is an $\SU(2)$ noncommutative principal fibration 
$\Sk^7 \to \St^4$ given in \cite{LS04}. Firstly, with $\lambda'_{a b} = e^{2 \pi \ii \theta'_{ab}}$ and $(\theta'_{ab})$ a real antisymmetric matrix, 
 the algebra $A(\Sk^7)$ of polynomial functions on the sphere $\Sk^7$ is generated by elements  
$\psi_a, \psi_a^*$, $a=1,\dots,4$, subject to relations
\beq\label{s7t}
\psi_a \psi_b = \lambda'_{a b} \, \psi_b \psi_a, \quad  \psi_a \psi_b^* = \lambda'_{b a} \, \psi_b^* \psi_a,
\quad \psi_a^*\psi_b^* = \lambda'_{a b} \, \psi_b^* \psi_a^* ,
\eeq
and with the spherical relation $\sum_a \psi_a^* \psi_a=1$. 
At $\theta=0$, it is the $*$-algebra of 
 complex polynomial functions on the sphere $S^7$. For the noncommutative Hopf bundle over the given 4-sphere $\St^4$, we need to select 
 a particular noncommutative 7 dimensional sphere $\Sk^7$. We take the one corresponding to the 
following deformation parameters
\beq\label{lambda7}
\lambda'_{ab}= 
\begin{pmatrix} 1 & 1 & \bar{\mu} & \mu \\ 
1 & 1 & \mu & \bar{\mu} \\
\mu & \bar{\mu} &1 & 1\\ 
\bar{\mu} & \mu &1 & 1 
\end{pmatrix}, \quad \mu = \sqrt{\lambda} \qquad \mathrm{or} \qquad
\theta'_{ab}=\frac{\theta}{2}\begin{pmatrix} 0 & 0 & -1 & 1 \\ 
0 & 0 & 1 & -1 \\
1 & -1 & 0 & 0 \\ 
-1 & 1 & 0 & 0  \end{pmatrix}.
\eeq
The previous choice is essentially the only one that  allows the algebra $A(\Sk^7)$ to carry an action 
 of the group $\SU(2)$ by automorphisms and such that the invariant subalgebra coincides with 
 $A(\St^4)$.  The best way to see this is by means of the matrix-valued function on $A(\Sk^7)$ (we are 
changing notations with respect to \cite{LS04}) 
\beq\label{Psi} 
\Psi  =
\begin{pmatrix} 
\psi_1 & - \psi^*_2 \\ 
\psi_2 & \psi^*_1 \\
\psi_3 & -\psi^*_4 \\
\psi_4& \psi^*_3
\end{pmatrix}.
\eeq
Then, the commutation relations of the algebra $A(\Sk^7)$, with deformation parameter in \eqref{lambda7}, gives that  $\Psi^\dagger \Psi = \II_2 $. As a consequence, the matrix-valued function $p = \Psi \Psi^\dagger$ is a projection,  $p^2=p=p^\dagger$, and its entries rather that functions in $A(\Sk^7)$ are (the generating) elements of $A(\St^4)$. Indeed, the right coaction of $A(\SU(2))$ on $A(\Sk^7)$ is simply given by
\beq \label{actionSU2}
\delta  (\Psi) = \Psi \pot  w , \qquad w = \begin{pmatrix} w_1 & - w^*_2 \\ w_2 & w^*_1 
\end{pmatrix} \in A(\SU(2)), \quad w w^\dagger = 1 = w^\dagger w .
\eeq
If $\sigma(a \ot b) = b \ot a$ is the flip, this gives 
$$ \delta  (\Psi^\dagger) = \sigma(w^\dagger  \pot \Psi^\dagger) ,
$$
and the invariance of the entries of $p$ follows at once: 
\beq\label{once}
p \mapsto \delta  (\Psi) \, \delta  (\Psi^\dagger) = p \pot  w w^\dagger = p \pot  1.
\eeq 
The generators of $A(S^4)$, the independent entries of $p$, are identified as bilinears expressions in the $\psi,\psi^*$'s. Explicitly, 
\beq\label{proj}
p = \Psi \cdot \Psi^\dagger =
\begin{pmatrix}
\zeta_0 & 0 & \zeta_1 & - \bar{\mu} \zeta_2^* \\
0 & \zeta_0 & \zeta_2  & \mu \zeta_1^* \\
\zeta_1^*& \zeta_2^* & 1-\zeta_0 & 0\\
-\mu \zeta_2 & \bar{\mu}   \zeta_1 & 0 & 1-\zeta_0
\end{pmatrix} , 
\eeq
with 
\begin{align} \label{sub}
\zeta_1 &=  \psi_1 \psi^*_3 + \psi^*_2 \psi_4 , \qquad
\zeta_2 =  \psi_2 \psi^*_3 - \psi^*_1 \psi_4 ,  \nn \\
 \zeta_0 &= \psi_1 \psi^*_1 + \psi^*_2 \psi_2 = 1 - \psi_3 \psi^*_3 - \psi^*_4 \psi_4 . 
\end{align}
By using the commutation relations of the $\psi$'s, one computes the commutation rules 
$\zeta_1 \zeta_2  = \lambda \zeta_2 \zeta_1$, $\zeta_1 \zeta_2^* = \bar{\lambda} \zeta_2^* \zeta_1$, 
and that  $\zeta_0$ is central and hermitian and $\zeta_1$, $\zeta_2$ are normal. The spherical relation for 
$\Sk^7$ gives an analogous one, 
$\zeta_1^* \zeta_1 + \zeta_2^* \zeta_2 = \zeta_0 (1-\zeta_0)$, for $\St^4$.

There are compatible toric actions on $\St^4$ and $\Sk^7$ (see e.g. \cite[\S  2.3]{Brain:2013wp}.)
With a slight change of notation, the torus $\IT^2$ acts on $A(\St^4)$ as 
\beq\label{eq:act-S4}
\sigma_s(\zeta_0, \zeta_1, \zeta_2) = (\zeta_0, e^{2\pi i s_1} \zeta_1, e^{2\pi i s_2} \zeta_2), \quad s\in \IT^2 .
\eeq
This action is lifted to a double cover action on $A(\Sk^7)$. The double cover map  $p: \wt \IT^2 \to \IT^2$ is  
given explicitly by $p:(s_1,s_2) \mapsto (s_1+s_2,-s_1+s_2)$.
Then $\wt \IT^2$ acts on the $\psi_a$'s as:
\beq \label{eq:lift-S7}
\wt \sigma: \left(  \psi_1, \psi_2, \psi_3, \psi_4 \right) 
\mapsto \left(e^{2\pi i s_1}~\psi_1, ~e^{-2\pi i s_1}~\psi_2, ~e^{-2\pi i s_2}~\psi_3, ~e^{2\pi i s_2}~\psi_4 \right) 
\eeq

The sense in which the algebra inclusion  $A(\St^4) \subset A(\Sk^7)$  is a nontrivial (faithfully flat)
noncommutative $\SU(2)$ principal bundle is explained in \cite{LS04}.  Here we mention 
that there is a canonical Galois maps $\chi : A(\Sk^7) \ot_{A(\St^4)} A(\Sk^7) \to A(S^{7}_{\theta})\otimes A(SU(2))$ which is invertible. 
The corresponding translation map $\tau : A(\SU(2)) \to A(\Sk^7) \ot_{A(\St^4)} A(\Sk^7)$  on generators is  
\beq
\tau(w) = \Psi^\dagger \pot_{A(\St^4)} \Psi 
\eeq
Indeed, $\chi \circ \tau(w) = \chi(\Psi^\dagger \pot_{A(\St^4)} \Psi) = \Psi^\dagger \delta (\Psi) = \Psi^\dagger \Psi \pot  
w = 1 \ot 1_2 w = 1 \ot w$. 

There is also a copy of the projection $p$ in the opposite algebra:
\beq
q =  \Psi \cdot_{op} \Psi^\dagger =
\begin{pmatrix}
\zeta_0 & 0 & \bar{\mu} \zeta_1 & - \zeta_2^* \\
0 & \zeta_0 &  \mu \zeta_2  & \zeta_1^* \\
 \mu \zeta_1^*& \bar{\mu} \zeta_2^* & 1- \zeta_0 & 0\\
- \zeta_2 & \zeta_1 & 0 & 1- \zeta_0
\end{pmatrix}
\eeq
The difference between $p$ and ${q}$ is due to the multiplication in $A(\Sk^7)$ versus the one in $A(\Sk^7)^{op}$. Indeed:
\beq\label{rojopprojp}
p_{mn} = \sum_r \Psi_{mr} \Psi^\dagger{}_{rn} , \qquad {q}_{mn} = \sum_r\Psi_{mr} \cdot_{op} \Psi^\dagger{}_{rn} = \sum_r \Psi^\dagger{}_{rn} \Psi_{mr} .
\eeq
With the commutation relations \eqref{s7t}, the condition $\Psi^\dagger \cdot \Psi = \II_2 $ leads also to $\Psi^\dagger \cdot_{op} \Psi = \II_2$.

\subsection{Principal bundles over even quantum spheres} \label{sec:-q-shp} 

Even noncommutative spheres $S_\theta^{2n}$, introduced in \cite{CL01}, were shown in \cite{Var01} to be homogeneous spaces of quantum groups $\SO_\theta(2n+1, \IR)$. 
The algebra of coordinate functions of the latter $A=\O(\SO_\theta(2n+1, \IR))$ is the total space algebra of a principal bundle over the algebra $B=\O(S_\theta^{2n})$ for the Hopf (structure) algebra 
$H=\O(\SO_\theta(2n, \IR))$. 
These bundles were worked out in details in \cite[\S 4.1.1]{ABPS17} that we follows with changes. 

Start with the commutative torus $\IT^n$ with generators $t_j, t_j^*$ and relations $t_j t_j^*=t_j^* t_j=1$.
Consider the bi-character $\gamma : \IT^n \times \IT^n \to \U(1)$ defined on generators by
$$
\gamma(t_j, t_k) = e^{\ii \pi \theta_{jk}} , \qquad \theta_{j k} = - \theta_{k j} .
$$
We shall denote $\lambda_{jk} = \gamma(t_j, t_k)^2 =e^{2i\pi \theta_{jk}}$. 
In order for the deformed algebra to still be a Hopf algebra one needs 
a left and a right action of $\IT^n = \diag(t_1, \dots, t_n, t_1^*, \dots, t_n^*)$ (or of 
$\IT^n \times  \IT^n$).
This action then allows one to deform the algebra $\O(\SO(2n))$ into 
an algebra $\O(\SOt(2n))$ described as follows. It has generators $\b{a} = (a_{jk})$, $\b{b} = (b_{jk})$, 
$\b{a}^* = (a_{jk}^*)$, $\b{b}^* = (b_{jk}^*)$ with commutation relations computed to be 
\begin{align}\label{thetaCR}
a_{ij}     a_{kl} & = \lambda_{ik}\lambda_{lj} ~a_{kl}    a_{ij} ,  \qquad
a_{ij}     b^*_{kl} = \lambda_{ki}\lambda_{lj} ~b^*_{kl}    a_{ij} \nn
\\
a_{ij}     b_{kl} & = \lambda_{ik}\lambda_{jl} ~b_{kl}    a_{ij}  ,  \qquad
a_{ij}     a^*_{kl} = \lambda_{ki}\lambda_{jl} ~a^*_{kl}    a_{ij} \nn
\\
b_{ij}     b_{kl} & = \lambda_{ik}\lambda_{lj} ~b_{kl}    b_{ij}  ,  \qquad
b_{ij}     b^*_{kl} = \lambda_{ki}\lambda_{jl} ~b^*_{kl}    b_{ij} 
\end{align}
together with their $*$-conjugated.  In fact, the Hopf algebra structure of $\O(\SO(2n))$ survives the quantization. 
In matrix notation the deformed $\O(\SOt(2n))$ has coproduct and counit given by
\beq\label{eq:mtxM}
M = (M_{J K}) = 
\begin{pmatrix}
\b{a} & \b{b} \\
\b{b}^* & \b{a}^*
\end{pmatrix} , \qquad \Delta(M) = M \pot M , \qquad \varepsilon(M) = \II  .
\eeq
To define an antipode there is a suitable determinant ${\det}_{\theta}(M)$ 
and one can pass to the quotient by the $*$-bialgebra ideal given by
\beq\label{idealM}
I_Q = <M^t Q M - Q , \, M Q M^t - Q , \, {\det}_{\theta}(M) - 1 > , \qquad
Q = \begin{pmatrix}
0 & \II_n \\
\II_n & 0 
\end{pmatrix} = Q^{-1}. 
\eeq
 The $*$-structure is then $*M=Q M Q$ while the antipode is $S(M)= Q M^t Q = M^\dagger $. 
 The previous conditions reads then $M^\dagger M = M M^\dagger = \II_{2n}$.
 
 The odd case of $\O(\SOt(2n+1))$ is defined in a similar fashion by deforming the 
 left and right actions of the torus 
 $\IT^n= \diag(t_1, \dots, t_n, t_1^*, \dots, t_n^*, 1)$ on $\O(\SO(2n+1))$. In matrix notation
 $$
N = (N_{J K}) = 
\begin{pmatrix}
\b{a} & \b{b} & \b{u} \\
\b{b}^* & \b{a}^* & \b{u}^* \\
\b{v} & \b{v}^* & x 
\end{pmatrix} , 
$$
with $n$-component column vectors $\b{u} = (u_j), \b{u}^* = (u_j^*)$ and row vectors $\b{v} = (v_j), \b{v}^* = (v_j^*)$ and a hermitian scalar $x$. The commutation relations are found to be given by
\beq\label{crN}
N_{I J} \, N_{K L} = \lambda_{IK} \lambda_{L J} N_{K L} \, N_{I J} . 
\eeq
Now the coproduct and antipode are as before by $\Delta(N) = N \pot N$ and $\varepsilon(N) = \II$ and one verifies  ideal conditions analogue to the ones in \eqref{idealM}:
\beq\label{ocN}
N^t Q N = Q , \quad N Q N^t = Q , \quad {\det}_{\theta}(N) = 1, \qquad 
Q = \begin{pmatrix}
0 & \II_n & 0 \\
\II_n & 0 & 0 \\
0 & 0 & 1 
\end{pmatrix} = Q^{-1}. 
\eeq
The $*$-structure is $*N=Q N Q$ while the antipode is $S(N)= Q N^t Q = N^\dagger$. Then the previous condution read $N^\dagger N = N N^\dagger = \II_{2n+1}$. 

The Hopf algebra $\O(\SOt(2n))$ is a quantum subgroup of $\O(\SOt(2n+1))$ 
with surjective Hopf algebra morphism
\begin{align}\label{subha}
\pi : \O(\SOt(2n+1)) & \to \O(\SOt(2n)), \nn \\
\begin{pmatrix}
\b{a} & \b{b} & \b{u} \\
\b{b}^* & \b{a}^* & \b{u}^* \\
\b{v} & \b{v}^* & x 
\end{pmatrix} & \mapsto  
\begin{pmatrix}
\b{a} & \b{b} & 0 \\
\b{b}^* & \b{a}^* & 0 \\
0 & 0 & 1 
\end{pmatrix} =: 
\begin{pmatrix}
{\bf h} & {\bf k} & 0 \\
{\bf k}^* & {\bf h}^* & 0 \\
0 & 0 & 1 
\end{pmatrix} = \b{w} .
\end{align}

\noindent
This results into a right coaction of $\O(\SOt(2n))$ on $\O(\SOt(2n+1))$:
\begin{align}
\delta^A : \O(\SOt(2n+1)) & \to \O(\SOt(2n+1)) \ot \O(\SOt(2n)) , \nn \\ 
\delta^A(N) & = N \pot \pi(N). 
\end{align}
The subalgebra $B$ of coinvariant elements, generated by the last column of the matrix $N$: $(u_j, u_j^*, x)$, 
is the algebra $\O(S_\theta^{2n})$ of coordinate functions on a quantum $2n$-sphere $S_\theta^{2n}$.
The commutation relations of the generators follows from \eqref{crN}:
\beq
u_i     u_j= \lambda_{ij} \, u_j     u_i \, , \qquad
u_i^*     u^*_j= \lambda_{ij} \, u^*_j     u^*_i \, , 
 \qquad u_i     u_j^*= \lambda_{ji} \, u_j^*  u_i \, ,
\eeq
and $x$ central. The orthogonality conditions \eqref{ocN} imply the sphere
relation 
$$
\sum_{j=1}^{n} 2 u_j^* u_j + x^2 = 1 ,    
$$
(each generator is normal $u_j^* u_j = u_j u_j^*$).
The algebra extension $\O(S_\theta^{2n}) \subset \O(\SOt(2n+1))$ is a Hopf Galois extension 
for the Hopf algebra $H=\O(\SOt(2n))$ (cf. \cite[\S 4.1.1]{ABPS17}). 
In particular we record the form of the translation map to be used later on.
In components
\begin{align} 
\label{eq:tranmap-h} 
\tau(\b{h}) &= \b{a}^\dagger \pot_B \b{a} + (\b{b}^*)^\dagger \pot_B \b{b}^* + \b{v}^\dagger \pot_B \b{v} , \nn \\
\tau(\b{k}) &= \b{a}^\dagger \pot_B \b{b} + (\b{b}^*)^\dagger \pot_B \b{a}^* + \b{v}^\dagger  \pot_B \b{v}^* .
\end{align}

\section{Algebraic $\theta$-deformations} \label{alg-def}
In this section we review the general scheme of deforming by the action of tori. 
This will be done in the crudest way via $\IZ^n$-graded spaces and deforming relevant structures 
by means of a bi-character. 
The role of $\IZ^n$ comes from it being the Pontryagin dual of the torus $\IT^n$ and one is effectively deforming 
objects with a torus action. More details are e.g. in \cite{Brain:2012} and \cite{Brain:2013wp}. In particular we shall deform principal bundles and associated Hopf algebroids. 
A general scheme of deformations of noncommutative principal bundles via convolution invertible 2-cocycles
$\gamma : H \ot H \to \IC$ on a Hopf  algebra $H$ is in \cite{ABPS17}.

Let $\mathcal T_n$ be the category of $\IZ^n$-graded complex vector spaces whose objects are written as (finite) sums of the kind
\begin{align*}
V = \bigoplus_{ r \in \IZ^n } V_r, \qquad
p_r : V \to V_r .
\end{align*}
Here $p_r$ is the projection onto the  $r$-th component, and most of the time  
we  simply use a subscript to indicate  the projection 
$ v_r = p_r (v)$ for  $v \in  V$.
Morphisms $\psi \in  \Hom(V , W)$ are linear maps that preserve homogeneity,
but not necessarily the degree. 
More precisely, there always exists a group homomorphism 
$\rho_\psi : \IZ^n \to \IZ^n$ such that
\begin{align}
\label{eq:mor-N} 
\psi (V_r) \subset V_{ \rho_\psi (r) }.  
\end{align}

In the $\theta$-deformation literature, one starts with a smooth 
action of a $n$-torus on a Fr\'echet  space $V$, $t \in  \IT^n \mapsto \alpha_t  \in  \mathrm{Aut} (V)$.
The induced $\IZ^n$-grading, based as mentioned on Pontryagin duality,  is given by
projections  $p_r : V \to V_r$, $r \in  \IZ^n$, taking the $r$-th
Fourier coefficients of the vector-valued function $ t \to \alpha_t (v)$,
\begin{align*}
p_r (v) = \int_{ \IT^n }   e^{- 2 \pi i r \cdot t} \alpha_t (v)\, \dd t  , \, \, \, \,  v \in  V,
\end{align*}
where $\dd t$ is the normalized  Lebesgue measure on  $\IT^n$.
Morphisms as in \eqref{eq:mor-N} corresponds to linear maps 
$ \widetilde \psi : V \to W$ which are $\IT^n$-equivariant up-to a 
group homomorphism $ \widetilde \rho_{ \widetilde \psi} : \IT^n \to \mathbb T^n$
so that the diagram commute:
\begin{equation}
\begin{tikzcd}
\mathbb T^n \times V 
\arrow[d, " \rho_{ \widetilde \psi} \times \widetilde \psi   "] \arrow[r] 
&  V \arrow[d, " \widetilde \psi"] \\
\mathbb T^n \times W \arrow[r]                   & W                   
\end{tikzcd} .
\end{equation}

The parameter $\theta$ in a  $\theta$-deformation is a 
$n \times  n$ skew-symmetric matrix and what is actually  needed for the
deformation is the induced bi-character on $\IZ^n$, that is a map 
\begin{align*}
\twcoc_\theta : \IZ^n \times  \IZ^n \to \IT , \quad 
(r, l) \mapsto  \twcoc_\theta (r , l) 
:=  e^{ \pi i \abrac{ \theta r  , l} } ,
\end{align*}
which is a $2$-cocycle in the sense of 
\begin{align}
\label{eq:2-coc-e-theta} 
\twcoc_\theta ( r , l) 
\twcoc_\theta ( r + l , s) =
\twcoc_\theta ( l , s) 
\twcoc_\theta ( r  , s + l ) , \quad r, s, l \in \IZ^n .
\end{align}

The tensor products functor
$ \otimes : \mathcal T_n \times  \mathcal T_n \to \mathcal T_n$
makes $ \mathcal T_n$ into a monoidal category, in which 
the $\IZ^n$-grading is assigned in the usual way,
\beq\label{tot-gra}
( V \otimes W)_s = \bigoplus_{ s = r + l }   V_r \otimes  V_l,  \qquad
s , r, l \in \IZ^n ,
\eeq
by taking the total degree of the natural bi-grading. 
One can deform the tensor functor via the following natural transformation 
$c^\theta$, for any $V , W \in  \mathcal T_n$,  
\begin{align}
c^\theta_{ V , W} :  V \otimes W \to V \otimes_\theta  W , \quad 
v_r \otimes w_l \to   v_r \otimes_\theta w_l := 
\twcoc_\theta (r, l) v_r \otimes w_l 
\label{eq:cVW-dfn} 
\end{align}
which is defined firstly on homogeneous elements and then extended by linearity.
It is not difficult to see that it has an inverse given by
$( c^\theta_{ V , W})^{-1} = c^{- \theta}_{ V , W}$.

Given an algebra $(A , m)$ in $ \mathcal T_n$, with multiplication
$m: A \otimes  A \to A$ preserving the grading,
\begin{align}
\label{eq:A-grd} 
m \brac{ A_r \otimes  A_l }  \subset A_{ r + l} , \, \, \, \, 
r, l \in \IZ^n,
\end{align}
its deformation $A_\theta = ( A , m_\theta )$ maintains the underlying 
(graded) vector space unchanged, but endowed with a  new multiplication: 
\begin{align}
\label{eq:m_theta-defn} 
m_\theta = m \circ c^\theta_{ A , A} : 
A \otimes  A \xrightarrow{ c^\theta_{ A , A}} A \otimes  A 
\xrightarrow{m} A .
\end{align}
As shown in \eqref{eq:cVW-dfn},
$m$ is twisted by a phase factor on homogeneous elements: 
\begin{align}
\label{eq:m_theta-r-l} 
m_\theta ( a_r , \tilde a_l )  = \twcoc_\theta (r , l)
m ( a_r  , \tilde a_l) ,
\end{align}
which, provided that  $m$ is commutative, leads to the commutation relations: 
\begin{align}
\label{eq:com-rltn-A} 
m_\theta ( a_r , \tilde a_l ) =
\twcoc_\theta (r , l)^2
m_\theta ( a_l , \tilde a_r ) .
\end{align}
The required associativity for $m_\theta$ follows directly
from the  $2$-cocycle condition in  \eqref{eq:2-coc-e-theta}.
For easy of notation in the following we shall denote $m_\theta ( a_l, \tilde a_r )= a_l \cdot_\theta \tilde a_r$.

Clearly, $\twcoc_\theta (r , \pm r) = 1$ since $\theta$ is skew-symmetric. 
We record this simple observation as a lemma which will be used often later on.
\begin{lem}
\label{lem:mtheta-m} 
For homogeneous element $a , \tilde a \in  A$ of the same degree 
or of the opposite degree, that is $\deg a \pm \deg \tilde a =0$, 
the deformed multiplication agrees with the original one:
\begin{align}
\label{eq:mtheta-m} 
m_\theta  (a , \tilde a) =  m( a , \tilde a) .
\end{align}
In particular, \eqref{eq:mtheta-m} holds whenever the product 
$m(a , \tilde a) \in A_{ 0 } $ belongs to the degree zero component, in this case, 
$a , \tilde a$ are not required to be homogeneous. 
\end{lem}
In a similar manner, for an $A$-module  $V$ in  $\mathcal T_n$ such that 
the action $ \rhd : A \otimes  V \to V$ preserves the grading as in 
\eqref{eq:A-grd}, the deformation
$ \rhd_\theta :=  \rhd \circ c^\theta_{ A , V}  $ makes $V_\theta$ into an
$A_\theta$-module.
The `associativity' (the action properties) for $\rhd_\theta$ again follows directly
from the  $2$-cocycle condition in \eqref{eq:2-coc-e-theta}. There is clearly a right-module version of this.

And finally, if $(C,\Delta)$ is a coalgebra in  $\mathcal T_n$ with $\Delta: C \to C\ot \IC$ that preserves the degree in the sense of \eqref{tot-gra}: $\Delta(c_s) = \sum_{r+l=s} \one{c_r} \ot \two{c_l}$ the deformation
$\Delta_\theta :=  c^{-\theta}_{C, C} \circ \Delta$ makes $C_\theta$ into a coalgebra with co-associativity again 
following from the  $2$-cocycle condition.

The next step in deforming a bialgebra (or even a Hopf algebra) structures needs some extra care. 
Also, for deforming a Hopf--Galois extension with structure Hopf algebra $H$, and aiming at including  
both examples in \S\S ~\ref{sec:HG-SU2} and \ref{sec:-q-shp}, 
it turns out that the construction of gradings on the algebra involved and the related assumptions are
quite different depending on whether the Hopf algebra is deformed or not. 

We will break the discussion into two scenarios to cover the constructions 
of both \S\S ~\ref{subsec:ScI} and \ref{subsec:ScII}
in which our aim is to get a (possible new) structure Hopf algebra with related comodule
algebras out of the  $\theta$-deformation scheme. 
After that, the deformation of the Ehresmann--Schauenburg bialgebroids can be handled
in a uniform way, and will be carried out in \S\S ~\ref{subsec:ES-alg} and \ref{subsec:flip-atpd}.

\subsection{Scenario I: No Hopf algebra is deformed}
\label{subsec:ScI} 

We start with a setting in which the Hopf algebra $H$ is not touched.
Thus, we assume that $H$ has trivial  $\IZ^n$-grading and a $H$-comodule algebra 
$A$ is $\IZ^n$-graded so that the multiplication preserves the grading as in 
\eqref{eq:A-grd}, 
and the coaction $\delta^A : A \to A \otimes  H$ also behaves the same way:
\beq
\label{eq:coprd-grd} 
\delta^A ( A_r ) \subset A_r \otimes  H .
\eeq 
Thus when writing $\delta^A(a) = \zero{a} \ot \one{a}$ one has $\deg a = \deg \zero{a} $.

The following is an almost free version of \cite[Cor. 3.16]{ABPS17}. 
\begin{prop} \label{prop:comodalg-theta}
Let $A_\theta = ( A , m_\theta)$ be the deformation of  $A$ as in 
\eqref{eq:m_theta-defn}. It is still a $H$-comodule algebra with 
the same coaction treated as 
$\delta^A: A_\theta \to A_\theta \otimes  H$.
Then, the coinvariant subspace $B = A^{\mathrm{co} H}$ remains the
same and the $\theta$-multiplication can be restricted onto  $B$ to form 
$B_\theta = (B, m_\theta) $.
Moreover, if the starting pair $(A , H)$ is a Hopf--Galois extension with algebra of coinvariant elements $B$,
such is its deformation $( H , A_\theta)$, with algebra of coinvariant elements $B_\theta$.
\end{prop}
\begin{proof}
The first part is evident. As for the final (almost evident) statement, consider the starting canonical map 
$\chi : A \ot _B A \longrightarrow A \ot H , \quad \chi (\tilde{a} \ot_B a) = \tilde{a} \zero{a} \ot \one{a}$ and define
$$
\chi_\theta : A_\theta \ot _{B_\theta} A_\theta  \longrightarrow A_\theta \ot H , \quad \chi_\theta (\tilde{a} \ot_B a) = \tilde{a} \cdot_\theta \zero{a} \ot \one{a}.
$$ 
Then, for $h \in H$ consider the starting canonical map $\tau(h) = \tuno{h} \ot_B \tdue{h}$ 
with components (sum of terms) of opposite degree $\deg \tuno{h} = - \deg \tdue{h}$ since 
$H$ has zero degree which is preserved by $\tau$. 
Then, from Lemma \ref{lem:mtheta-m},  \begin{align*}
\chi_\theta (\tuno{h} \ot_B \tdue{h}) & = \tuno{h} \cdot_\theta \zero{\tdue{h}} \ot \one{\tdue{h}} = 
\tuno{h} \zero{\tdue{h}} \ot_B \one{\tdue{h}} \\
& = \chi(\tuno{h}\ot_B \tdue{h}) 
\end{align*}
(the latter being just $1 \ot h$ from \eqref{p7}) and $\chi_\theta$ is invertible if and only if $\chi$ is.

Thus the translation map of $\chi_\theta$ is the same as the starting 
undeformed one that can be considered as a map 
$\tau : H \to A_\theta \ot_{B_\theta}  A_\theta$.
 \end{proof}

\begin{rem}[On the degree of the translation map]
The fact that, in writing for the translation map $\tau(h) = \tuno{h} \ot_B \tdue{h}$, one can take 
$\deg \tuno{h} = - \deg \tdue{h}$ does not depend on the representatives: suppose 
$\tuno{h} \ot_B \tdue{h} = \tuno{\tilde h} b \ot_B \tdue{h} = 
\tuno{\tilde h} \ot_B b \tdue{h} = \tuno{\tilde h} \ot_B \tdue{\tilde h}$. Then it follows that 
$\deg \tuno{h} = - \deg \tdue{h}$ if and only if $\deg \tuno{\tilde h} = - \deg \tdue{\tilde h}$.
\end{rem}

\begin{exa}[Noncommutative Hopf-fibration] \label{rem-dtm}
As mentioned, the $\IZ^n$-grading we consider is derived from a torus action.
To construct the $\SU(2)$-fibration $S^7_\theta \to S^4_\theta$ in \S\ref{sec:HG-SU2},
one begins with a two torus action defined in \eqref{eq:act-S4} and \eqref{eq:lift-S7},
in which all generators in \ref{s4t} and \ref{s7t} are $\IT^2$-eigenfunctions.  
Then, 
\begin{align*}
& \deg \zeta_0 = \deg \zeta_0^* = 0
\\
&  \deg\zeta_1 = (1, 0), \, \, \, \, \deg \zeta_2 = (0 , 1)
\\
& \deg \psi_1 =   - \deg \psi_2  
= (1,0)  ,  \, \, \, \, 
\deg \psi_4 =   -\deg \psi_3   
= (0,1)  .
\end{align*}
The deformation matrix just reads  
$\begin{bmatrix} 0& - \theta \\ \theta & 0
\end{bmatrix} $, with $\theta \in \mathbb R$. 
We have, according to   \eqref{eq:m_theta-r-l},
\begin{align*}
\zeta_\mu \cdot_\theta \zeta_\nu = \sqrt{ \lambda_{ \mu \nu}}
\zeta_\mu \zeta_\nu, \, \, \, \, 
\psi_a \cdot_\theta \psi_b = \sqrt{ \lambda'_{a b}}
\psi_a \psi_b ,
\end{align*}
so that the commutation relations \ref{s4t} and \ref{s7t}
follow immediately from \eqref{eq:com-rltn-A}.
It is also worth noting that the double covering between
the two actions \eqref{eq:act-S4} and \eqref{eq:lift-S7}
is exactly dual to the following map $\IZ^2 \to \IZ^2$: 
\begin{align*}
(1,0) &= \deg \zeta_1 \mapsto (1, 1) = \deg  \psi_1 \psi^*_3 
= \deg \psi_2^* \psi_4, \\
(0,1) &= \deg \zeta_2 \mapsto (-1,1) = \deg  \psi_2 \psi^*_3 
= \deg \psi_1^* \psi_4.
\end{align*}
revealed  in the  embedding $A(S^4_\theta) \to A(S^7_\theta)$ given by 
\ref{sub}.
\qed\end{exa}

The details of the construction of the Ehresmann--Schauenburg bialgebroid related to the Hopf--Galois extension 
$(H, A_\theta)$, are postponed to \S\ref{subsec:ES-alg}. 

\subsection{Scenario II: Deforming Hopf algebras and homogeneous spaces}
\label{subsec:ScII} 

Unlike the previous section, in order to deform a Hopf algebra $H$
(or in a more accurate context, to only deform the algebra structure of $H$),
in a way that the all compatibilities axioms for Hopf algebras remains,
one needs a more delicate setup for the grading and the $\theta$-matrix.
To motivate the long list of requisites below,  
the reader is referred to App. \ref{sec:Qgrp}, where we recall the
original formulation in terms of torus actions due to Rieffel
\cite{Rieffel:1993tw}.

Let  $H = \bigoplus_{ r, s \in \IZ^n} H_{ (r,s)} $
be a Hopf algebra with a bi-grading of $\IZ^n$ (in particular, a grading of $\IZ^{2n}$), 
such that the group homomorphism $\rho_{ \psi}$ on
gradings in \eqref{eq:mor-N} induced via the structure maps of $H$ are given as follows:
\begin{itemize}
\item[i)] the multiplication preserves the grading as in \eqref{eq:A-grd}
\begin{align}
\label{eq:prd-bigrd} 
m ( H_{ (r ,s)} \otimes  H_{ (p , q)} )
\subset H_{ (r + p , s + q)} \, , 
\end{align}
\item[ii)] the coproduct $\Delta : H \to H \otimes  H$, is required to be such that
\begin{align}
\label{eq:coprd-bigrd} 
\Delta \brac{  H_{ (r, l )}}  \subset 
\bigoplus_{ s \in \IZ^n} H_{ (r , s)} \otimes  H_{ (s, l)}  \, ,
\end{align}
\item[iii)]
the counit factors through the projection:
\begin{align}
\label{eq:counit-bigrd} 
\varepsilon: H \to \bigoplus_{ s \in \IZ^n}  H_{ (s , s)}  \to \mathbb{C},
\end{align}
that is $ \varepsilon ( h_{ (r , l)} ) = 0$ for all homogeneous elements 
$h_{ (r , l)} $ with $r \neq l$ ,

\item[iv)] for the antipode and the $*$-operator (if $H$ has one), one assumes  
\begin{align}
\label{eq:antp-star-gd} 
S ( H_{ (r , l )} )  \subset H_{ (-l , - r) } ,
\qquad
* ( H_{ (r , l)} )  \subset H_{ (-r , -l) } .
\end{align}
\end{itemize}

\begin{rem}[On condition \eqref{eq:coprd-bigrd}]
From the general assigning of the total degree in  \eqref{tot-gra}, on the right hand side of \eqref{eq:prd-bigrd} one would have 
$\bigoplus_{ a+c = r, b+d = l } H_{ (a , b)} \otimes  H_{ (c, d)}$.
The subspaces
$\bigoplus_{ s \in \IZ^n }  H_{ (r , s)} \otimes  H_{ (s, l)} $ of 
$H \otimes  H$ when summed on the indices $r,l$ corresponds to 
the subspace $ \mathcal D$ in \eqref{eq:subspace-C} in which the coproduct lands.
\qed\end{rem}

Next, 
let $\theta$ be  a $n \times n$ skew-symmetric matrix and put
\begin{align}
\label{eq:Theta-matrix} 
\Theta = 
\begin{bmatrix}  \theta   & 0 \\   0 & - \theta   \end{bmatrix} ,
\end{align}
so that their $2$-cocycles are related as follows:
for $r  = (r_1 , r_2 )$ and $l = (l_1 , l_2)$,
\begin{align*}
\twcoc_{\Theta} ( r , l )
= 
\twcoc_\theta ( r_1 , l_1  )
\twcoc_{ -\theta  }  ( r_2 , l_2). 
\end{align*}
Denote by $H_{\Theta} = (H , \cdot_{\Theta} )$ the deformed algebra, with the new 
multiplication given, on homogeneous elements $h_r , g_l $ of degree 
$r , l \in  \IZ^n$ respectively, by
\begin{align}
\label{eq:Theta-hr-gl} 
h_r \cdot_{\Theta} g_l  = 
\twcoc_\Theta (r ,l)  
h_r g_l  =
\twcoc_\theta ( r_1 , l_1  )
\twcoc_{ - \theta }  ( r_2 , l_2) 
h_r g_l  .
\end{align}

\begin{lem}
\label{lem:coprd-prd} 
With the condition in \eqref{eq:coprd-bigrd}, 
the (undeformed) coproduct $\Delta$ is still an algebra homomorphism for the product $ \cdot_{\Theta}$: 
\begin{align*}
\Delta ( h \cdot_{\Theta} g) = 
\one{h} \cdot_\Theta \one{g} \otimes \two{h} \cdot_\Theta \two{g}.
\end{align*}
\end{lem}
\begin{proof}
It suffices to work with homogeneous elements. 
Take $h , g \in H$, with $\deg h = (r , l)$ and  $\deg g = (p , q)$ with  
their components in the coproduct, $\Delta x = \one{x} \ot \two{x}$ in Sweedler
notation, 
that can be assumed to be homogeneous as well:
\begin{align*}
\deg \one{h} = (r , s) , \quad  
\deg \two{h} = (s , l)  , \quad  
\deg \one{g} = (p , k)   , \quad   
\deg \two{g} = (k , q) .
\end{align*}
where only $s, k$ vary within to the components.
Then, 
\begin{align*}
\brac{ \one{h} \otimes  \two{h} } \cdot_\Theta
\brac{ \one{g} \otimes  \two{g} }
& =
\one{h} \cdot_\Theta \one{g} \otimes  \two{h} \cdot_\Theta \two{g}
\\
& =  
\twcoc_\theta ( r, p )
\twcoc_{-\theta} (s  , k)
\twcoc_{\theta} (s  , k)
\twcoc_{ - \theta} ( l, q)
\one{h}  \one{g}   \two{h}  \two{g} 
\\
& =  
\twcoc_{ - \theta} ( l, q)
\twcoc_\theta ( r, p )
\one{h}  \one{g}   \two{h}  \two{g} \\
& =
\twcoc_{ - \theta} ( l, q)
\twcoc_\theta ( r, p )
\Delta ( h g)  \\
& =   
\Delta ( h \cdot_\Theta g) ,
\end{align*}
as stated. \end{proof}

\begin{prop}
\label{prop:H-Theta} 
By  $\theta$-deforming the multiplication of $H$ as in \eqref{eq:Theta-hr-gl},
we obtain a new Hopf algebra
$H_\Theta = (H , \cdot_\Theta , \Delta , \varepsilon , S)$ 
with the same coproduct, counit and antipode.
\end{prop}
\begin{proof}
The compatibility  between the algebra $ \cdot_\Theta$ and the coalgebra 
$\Delta$ structures has been dealt with  in Lemma \ref{lem:coprd-prd}.
The coproduct $\Delta$ and counit $\varepsilon$ are not deformed at all, 
thus property
$( \varepsilon \otimes 1) \Delta = 1 = ( 1 \otimes \varepsilon ) \Delta$
remains. We are left to verify 
\begin{align}
\label{eq:S-ve-2chk} 
S ( \one{h}) \cdot_\Theta \two{h} = \varepsilon ( h ) = 
\one{h} \cdot_\Theta S ( \two{h}) , \quad \forall  h \in  H.
\end{align}
Suppose $h$, $\one{h}$ and $\two{h}$ are homogeneous of degree  $(r , l)$, 
$(r , s)$ and  $(s , l)$ respectively. By the assumptions in 
\eqref{eq:antp-star-gd}, $S ( \one{h})$ is of degree $(-s , -r)$, thus
\begin{align*}
S ( \one{h}) \cdot_\Theta \two{h} 
& =      
\twcoc_\theta (-s ,s) \twcoc_{ - \theta } ( r , l)
S ( \one{h}) \cdot \two{h} =      
\twcoc_{ - \theta } ( r , l)
S ( \one{h}) \two{h}  \\
& =     
\twcoc_{ - \theta } ( r , l)
\varepsilon ( h) = \varepsilon ( h ) .
\end{align*}
For the last step, we need  to invoke \eqref{eq:counit-bigrd}, so that 
$ \varepsilon (h) = 0$ whenever $l \neq r$, while for  $r = l$, we have 
$ \twcoc_{ - \theta } ( r , l) = 1$.
\end{proof}

Next, 
let $ \mathcal M^H$ be the category of $H$-comodule with a bi-grading of  $\IZ^n$ 
and such that the coaction $\delta^V : V \to V \otimes  H$ with $V \in  \mathcal M^H$, 
behaves in a similar way  to the coproduct in
\eqref{eq:coprd-bigrd} as regarding the grading:
\begin{align}
\label{eq:coact-bigrd} 
\delta^V \brac{   V_{ (r , l)}}  \subset \bigoplus_{ s \in  \IZ^n}    
V_{ (r ,s) } \otimes  H_{ ( s , l)}   .
\end{align}
The co-representations $ \mathcal M^{H_\Theta}$ of $H_\Theta$,
keep the same objects and morphisms as $ \mathcal M^H$.
Modification  only occurs on the coaction on the monoidal
structure.  Namely, in the coaction on $V \otimes  W$, where $V , W$ are in $ \mathcal M^H$, we must 
use of  the multiplication of $H_\Theta$: 
\begin{align}
\label{eq:coact-HTheta} 
\delta^{V \otimes_\Theta W}: V \otimes W \to   V \otimes W \otimes  H_\Theta,
\quad v \otimes  w \mapsto 
\zero{v} \otimes \zero{w} \otimes 
\one{v} \cdot_\Theta \one{w} .
\end{align}

When deforming a comodule algebra $A$ in $ \mathcal M^H$, which play 
the role of function algebra on the noncommutative principal bundle,
we have to impose similar conditions. That is, we have that
$A$ also admits a bi-grading of  $\IZ^n$ such that 
\begin{enumerate}
\item the product of $A$ preserves the bi-grading as in \eqref{eq:prd-bigrd};
\item[]
\item the coaction  $\delta^A : A \to A \otimes  H$ satisfies
\eqref{eq:coact-bigrd} on the bi-grading. 
\end{enumerate}
The first condition allows one to form the deformed algebra $A_\Theta$ and the 
second one makes sure we still have a comodule algebra after deformation.

\begin{prop}
\label{prop:comodalg-Theta} 
Consider a Hopf--Galois extension $(A , H)$ with both $H$ and $A$ endowed with a bi-grading of $\IZ^n$, and algebra of coinvariants $B= A^{co H}$ (with a heredity bi-grading from $A$). Then,    
their bi-grading leads to the deformed algebras 
$H_\Theta = ( H , \cdot_\Theta )$  and $A_\Theta = ( A , \cdot_\Theta)$
according to  \eqref{eq:Theta-hr-gl}.
Moreover, $A_\Theta$ is a $H_\Theta$-comodule  algebra with the same coaction 
viewed as a map $\delta^A_\Theta : A_\Theta \to  A_\Theta \otimes  H_\Theta$.
Also, the coinvariant subspace  
$B_\Theta = A_\Theta^{\mathrm{co} H_\Theta } = (A^{\mathrm{co}H} , \cdot_\Theta) = (B, \cdot_\Theta)$
maintains its starting vector space sitting inside $A_\Theta$ as a subalgebra. 
\end{prop}
\begin{proof}
Observe that, from \eqref{eq:coprd-bigrd} and \eqref{eq:coact-bigrd}, the coaction  
$\delta^A$ and coproduct  $\Delta$ of  $H$ change the bi-grading in a similar manner, 
hence  compatibility between the coaction and multiplication of $A$ can be
proved  along the lines of Lemma \ref{lem:coprd-prd}.

Since the coaction is taken directly from $(A , H)$, the coinvariant
subspace remains the same as a vectors space. Moreover,
the $\theta$-multiplication differs from the original one by a phase factor 
on homogeneous elements,  thus it maps  $B \otimes B$ into $B$. 
In other words,  $ \cdot_\Theta$ can be restricted onto the coinvariant
subspace $B$ to form $B_\Theta$.
\end{proof}

\begin{exa}
\label{eg:Q-SO}
Let us specialize the discussion in Appendix \ref{sec:Qgrp} to the case 
$G = \SO(2n)$ and $ \widetilde G = \SO ( 2n +1 ) $  and 
discuss the bi-grading behind the quantum spheres in \eqref{sec:ev-q-shp}
in great detail.
The torus action $\alpha$ of $\IT^n \times  \IT^n $ is now given by matrix multiplications
from two sides so that all the generators in \eqref{thetaCR} and \eqref{crN}
are eigenfunctions:
\begin{align}
\label{eq:eigfuns-eigvals} 
\begin{split}
\alpha_{t, \tilde t} ( a_{ j k } )  = t_j \tilde t_k a_{  jk} , 
\, \, \, \, 
\alpha_{t, \tilde t} ( b_{ j k } )  = t_j \tilde t_k^* b_{  jk} ,
\, \, \, \, 
\alpha_{t, \tilde t} ( u_{ j  } )  =  t_j u_{  j} ,
\, \, \, \, 
\alpha_{t, \tilde t} ( v_{ k  } )  =  \tilde t_k v_{  k} ,
\\
\alpha_{t, \tilde t} ( a^*_{ j k } )  = t_j^* \tilde t^*_k a_{  jk} , 
\, \, \, \, 
\alpha_{t, \tilde t} ( b^*_{ j k } )  = t_j^* \tilde t_k^* b_{  jk} ,
\, \, \, \, 
\alpha_{t, \tilde t} ( u^*_{ j  } )  =  t_j^* u_{  j} ,
\, \, \, \, 
\alpha_{t, \tilde t} ( v^*_{ k  } )  =  \tilde t_k^* v_{  k} .
\end{split}
\end{align}
Therefore, we can reconstruct the algebras 
$ H_\Theta = \mathcal O (\SO_\theta (2n))$ and
$ \widetilde H_\Theta = \mathcal O (\SO_\theta (2n + 1))$  by
assigning the following degrees to generators:  
\begin{align}
\begin{split}
& \deg a_{ i j } = (e_i , e_j) = - \deg a^*_{  i j} , 
\, \, \, \, 
\deg b_{ i j } = (e_i , - e_j) = - \deg b^*_{  i j}, 
\\
& \deg u_i = - \deg u_i^*  = ( e_i, 0), 
\, \, \, \, 
\deg v_i = - \deg v_i^*  = (0 , e_i) ,
\end{split}
\label{eq:gnt-degs}
\end{align}
where $\set{e_j, \, j =1, \cdots, n}$ is the standard basis of  $\IZ^n$, 
and extents to the whole algebra according to \eqref{eq:prd-bigrd}.
For homogeneous elements, the new multiplication differs from the commutative one by 
the phase factors as in \eqref{eq:Theta-hr-gl} instance, 
\begin{align*}
a_{ i j } \cdot_\theta   a_{ k l }  = 
\twcoc_{\theta} ( e_i , e_k )
\twcoc_{ - \theta }  ( e_j , e_l ) 
a_{ i j } a_{ k l }    =
\sqrt{ \lambda_{ i k} \lambda_{l j}  }
a_{ i j } a_{ k l }   ,
\end{align*}
and similarly, 
$  a_{ i j } \cdot_\theta   b_{ k l }  =
\sqrt{ \lambda_{ i k} \lambda_{ j l}  } b_{ i j } b_{ k l }$, while
for generators $\mathbf u$ and  $ \mathbf v$ in \eqref{crN},
\begin{align*}
u_i \cdot_\theta u_ j =
\twcoc_\theta (e_i , e_j)   u_i u_j  =
\sqrt{ \lambda_{ i j} }    u_i u_j   ,
\\
v_i \cdot_\theta v_ j =  
\twcoc_{-\theta} (e_i , e_j)v_i v_j = 
\sqrt{ \lambda_{ j i} }      v_i v_j .
\end{align*}
One recovers the commutation relations in \eqref{thetaCR} and \eqref{crN}
by taking \eqref{eq:com-rltn-A} into account.

Let us sample the assumptions 
\eqref{eq:coprd-bigrd} -  \ref{eq:antp-star-gd} on 
some of generators. For the coproduct:
\begin{align*}
\Delta ( a_{ j l} )  =
\sum_{ s } a_{ j s} \otimes  a_{ s l} + b_{ j s} \otimes  b^*_{ s l} ,    
\end{align*}
the right hand side indeed fulfils
$ a_{ j s} \otimes  a_{ s l} \in  H_{ (j ,s)} \otimes H_{ ( s , l)}$
and 
$   b_{ j s} \otimes  b^*_{ s l} \in  H_{ ( j , -s)} \otimes  H_{ (-s , l)}$.
For the counit  $\varepsilon$ defined by $ \varepsilon (N) = \mathbb I$, 
only the diagonal entries of $N$ will survive after applying  $\varepsilon$ 
and they indeed belong to
$\bigoplus_{ s \in \IZ^n } H_{ (s,s)}$ as required in \eqref{eq:counit-bigrd}. 
For the $*$-operator $*N = Q N Q$, we see, for instance, 
that $ \deg ( b_{  i j} )^* = \deg b^*_{  i j}  = - \deg b_{  i j} $.
For the antipode $ S (N) = N^{\dagger}$, we would like to check on, say $u_j$: 
$\deg u_j = ( e_j , 0)$ compared with
$ \deg S ( u_j ) = \deg v^*_j = (0, - e_j)$.
Lastly, the analysis of \eqref{eq:coact-bigrd} for the coaction $\delta^{ \widetilde H}$ is just the same as that for the coproduct $\Delta$.
\qed\end{exa}

\begin{exa}
The matrix representation of $\SO(2n)$ in \eqref{eq:mtxM} and \eqref{idealM} 
require an extra structure on
$\IR^{2n}$, that is a choice of polarization. Concretely, one identifies
$\mathbb{R}^{2n} \cong \mathbb{C}^n$ and choose a basis formed by complex
coordinates $\set{z_j, \bar z_j, \, j = 1 \cdots n}$, with respect to which
the coefficient matrix of the Euclidean  inner product is of the form $Q$ in
\eqref{idealM}. We recall a remark made in \cite[\S 8]{Connes:2002wh} 
which further motivates the bigrading setting of our scenario II in connection with 
the grading in scenario I. 

At the level of the endomorphisms $M(2n,\mathbb{R})$, 
the identification is achieved by realising 
$ M(2n, \mathbb{R}) \subset \mathrm{End}  ( \mathbb{C}^n) 
\cong (\mathbb{C}^n)^* \otimes \mathbb{C}^n$. 
One first applies the $\theta$-deformation to
$(\mathbb{C}^n)^* \otimes \mathbb{C}^n$ following the setting in scenario I,
which gives rise to two deformed algebra 
$A( \mathbb{R}^{2n}_{ \theta}) = A(\mathbb{C}^n_{ \theta}) $
and $A(\mathbb{R}^{2n}_{ -\theta}) = A(\mathbb{C}^n_{ -\theta})$, with 
generators  
$\set{z^j , \bar z^j := (z^j)^* , \, j=1, \cdots, n}$ for $ A(\mathbb{R}^{2n}_{ \theta})$ and
$\set{z_j , \bar z_j := (z_j)^*, \, j=1, \cdots, n}$ for $ A(\mathbb{R}^{2n}_{ -\theta})$.
The $n$-torus action  $\alpha$ is the standard one:
\begin{align}
    \label{eq:z_j-actn} 
\alpha_t ( z_j ) = t_j  z_j , \, \,
\alpha_t ( \bar z_j ) = \bar t_j  \bar z_j , \qquad
\alpha_t ( z^j ) = t_j  z^j , \, \,
\alpha_t ( \bar z^j ) = \bar t_j  \bar z^j , \, \,
\end{align}
 which leads to the $\IZ^n$-grading:
\begin{align}
    \label{eq:z_j-deg} 
e_j = \deg z^j = \deg z_j = -  \deg \bar z^j = - \deg \bar z_j,
\end{align}
where $\set{e_j, \, j=1, \cdots, n}$ is the standard basis of $\IZ^n$. 
The algebra structure of $A(\mathbb{R}^{2n}_{ \theta })$ is determined by
the commutation relations: 
\begin{align*}
    z_j z_k = \lambda_{ j k} z_k z_j, \, \quad 
    \bar z_j z_k = \lambda_{  k j} z_k \bar z_j .
\end{align*}
Those for $A(\mathbb{R}^{2n}_{ - \theta })$ are obtained by replacing 
$\lambda_{ j k}$ with $ \bar \lambda_{ j k}$.
Now, the deformed $*$-algebra 
$\mathcal O(M_\theta (2n , \mathbb{R}))$ 
(One can forget the coalgebra structure for the time being.)
has already been defined in \eqref{eq:mtxM} and \eqref{thetaCR} 
in terms of generators and relations.
A key point is that there is
a $*$-algebra homomorphism $\varphi: \mathcal O(M_\theta (2n , \mathbb{R})) \to
A(\mathbb{R}^{2n}_{ \theta})  \otimes A(\mathbb{R}^{2n}_{  - \theta})$
induced by
\begin{align*}
    \varphi (a_{ i j}) = z^i \otimes z_j , \, \quad
    \varphi (b_{ i j}) = z^i \otimes \bar z_j .
\end{align*}
Furthermore, the map $\varphi$ is injective and transfers the torus action 
$\alpha \otimes \alpha $ (cf. \eqref{eq:z_j-actn}),
or equivalently, the bi-grading of 
$ A(\mathbb{R}^{2n}_{ \theta})  \otimes A(\mathbb{R}^{2n}_{ -\theta})$
(cf. \eqref{eq:z_j-deg}), to those described in the Example \ref{eg:Q-SO}: see 
\eqref{eq:eigfuns-eigvals} and \eqref{eq:gnt-degs}.
\qed
\end{exa}

Let us now take a closer look at the algebra of coinvariants and at the balanced product. 

\begin{lem} \label{lem:balancedts-deg} 
With the assumptions on $H$ and  $A$ as before, 
the coinvariant subalgebra $B = A^{ \mathrm{co} H}$ 
is contained in
\begin{align}
\label{eq:B-deg} 
B \subset \bigoplus_{ r \in \IZ^n }  A_{ (r , 0) } .
\end{align}
\end{lem}
\begin{proof}
The (algebra of functions on) the torus $\IT^n$ acting on the right is contained in $H$ and gets washed away when passing to the coinvariant elements for the coaction of $H$. 
Explicitly, consider a homogeneous element $b \in  B$ with $\deg b = (r, l)$. 
From \eqref{eq:coact-bigrd} we have 
$\deg b_{ (0)} = (r , s) $ and $\deg b_{ (1)} = (s , l)$ where $s \in  \IZ^n$
depends on the components. 
The condition of being coinvariant
$ b_{ (0)} \otimes  b_{ (1)} = b \otimes  1 $ forces that 
$(r, s ) = ( r , l)$ and $ (s, l) = (0 , 0)$, hence $\deg b$ is  always 
of the form $(r , 0)$ for some  $r \in  \IZ^n$.
\end{proof}
The  $2n$-sphere $B = \mathcal O (S^{2n})$ and its deformation
in \eqref{eg:Q-SO} indeed satisfy \eqref{eq:B-deg}: 
the generators  of $B$ (or $B_\Theta$) are
$\set{u_j , u^*_j , \, j = 1, \cdots, n}$ which are of degree $(\pm e_j ,0)$.

When forming the balanced tensor product 
$a \otimes_B \tilde a$, where $a , \tilde a \in A$,
the degrees  of $a $ and $ \tilde a$ (assumed to be homogeneous)
depend on the choice of the representative. However, from the previous lemma, 
the action of $B$ only varies the left degree.
As we shall see in next lemma, 
by slightly abusing the notation, for the translation map we can write
\begin{align}
\label{eq:tau-bigrd-Theta} 
\tau \brac{ H_{ (r , l)}  }  \subset \bigoplus_{ p \in \IZ^n}
A_{ ( - p , - r)} \ot_B A_{ (p , l)} .
\end{align}
Let us check this on $\tau (\mathbf h)$ and $\tau (\mathbf k)$ in \eqref{eq:tranmap-h}. 
For the $(r,l)$-entry of  $\mathbf h$, we have
\begin{align}\label{tauhexp}
\tau ( h_{ r l} ) 
& =
\sum_{ s } (\mathbf a)_{ r s}^{\dagger} \ot_B \mathbf a_{ s l} + 
(\mathbf b^*)_{ r s}^{\dagger} \ot_B (\mathbf b^*)_{ s l} +
(\mathbf v)^{\dagger}_r \ot_B \mathbf v_l 
\nn \\
& =
\sum_{ s }
a^*_{ s r} \ot_B a_{ s l }  +
b_{ s r} \ot_B  b^{*}_{ s l}  +
v^*_{ r }  \ot_B  v_l .
\end{align}
We see that
$ a^*_{ s r} \otimes a_{ s l }
\in \widetilde H_{ (-e_s ,-e_r)} \otimes \widetilde H_{ ( e_s ,e_l )} $,
$ b_{ s r}  \otimes  b^{*}_{ s  l}  
\in \widetilde H_{ (e_s ,-e_r)} \otimes \widetilde H_{ ( -e_s ,e_l )} $
as well as 
$ v^*_{ r }  \otimes  v_l
\in \widetilde H_{ (0 ,-e_r)} \otimes \widetilde H_{ ( 0 ,e_l )} $
all satisfy \eqref{eq:tau-bigrd}.
Similarly, for the $(r,l)$-entry of  $\mathbf k$, we have
\begin{align}\label{taukexp}
\tau ( k_{ r l} ) 
& =
\sum_{ s } (\mathbf a)_{ r s}^{\dagger} \ot_B \mathbf b_{ s l} + 
(\mathbf b^*)_{ r s}^{\dagger} \ot_B (\mathbf a^*)_{ s l} +
(\mathbf v)^{\dagger}_r \ot_B \mathbf v^*_l  
\nn \\
& =
\sum_{ s }
a^*_{ s r} \ot_B b_{ s l }  +
b_{ s r} \ot_B  a^{*}_{ s l}  +
v^*_{ r }  \ot_B  v^*_l .
\end{align}
We see that
$ a^*_{ s r} \otimes b_{ s l }
\in \widetilde H_{ (-e_s ,-e_r)} \otimes \widetilde H_{ ( e_s , -e_l )} $,
$ b_{ s r}  \otimes  a^{*}_{ s  l}  
\in \widetilde H_{ (e_s ,-e_r)} \otimes \widetilde H_{ ( -e_s , -e_l )} $
as well as 
$ v^*_{ r }  \otimes  v^*_l
\in \widetilde H_{ (0 ,-e_r)} \otimes \widetilde H_{ ( 0 , -e_l )} $
and again they all satisfy \eqref{eq:tau-bigrd}.
We also have
\begin{align*}
    m_\Theta ( \tau (\mathbf h )) & =
    \mathbf a^{\dagger} \cdot_\Theta \mathbf a +
    \mathbf b^{t} \cdot_\Theta \mathbf b^* +
    \mathbf v^{\dagger} \cdot_\Theta \mathbf v = \varepsilon( \mathbf h ) \II = \II, \\
    m_\Theta ( \tau (\mathbf k )) & =
    \mathbf a^{\dagger} \cdot_\Theta \mathbf b +
    \mathbf b^{t} \cdot_\Theta \mathbf a^* +
    \mathbf v^{\dagger} \cdot_\Theta \mathbf v^* = \varepsilon( \mathbf k ) \II = 0 
\end{align*}
and both agree with \eqref{p5} and Corollary \ref{cor:tau-bigrd} below. 

\begin{lem} \label{lem:balancedts-deg-2} 
Let $(A , H)$ be a Hopf--Galois extension fulfilling 
all assumptions of earlier.
For any homogeneous elements $h \in  H_{ (r , l)} $, there are suitable
representatives for the translation map $\tau ( h )  = h^{\abrac{ 1 }} \ot_B h^{\abrac{ 2 }}$ 
such that
\beq
\label{eq:tau-bigrd} 
\deg h^{\abrac{ 1 }}  = ( -p , - r )   , \, \, \, \,   
\deg h^{\abrac{ 2 }}      = ( p , l) ,
\eeq
where, by taking \eqref{eq:B-deg} into account, the left degree $p$ depends on 
the components $ h^{\abrac{ 1 }},  h^{\abrac{ 2 }}$ and the choice 
of the representatives (cf. also Remark \ref{rem-dtm}).
\end{lem}
\begin{proof}
The  constraint on degrees in \eqref{eq:tau-bigrd}  follows from \eqref{p7}:
$ h^{\abrac{ 1 }} (h^{\abrac{ 2 }})_{ (0)}  
\otimes (h^{\abrac{ 2 }})_{ (1)} = 1 \otimes h $.
Suppose $\deg h = (r , l)$, $ \deg h^{\abrac{ 2 }} = (p ,q)$ 
and $\deg h^{\abrac{ 1 }} = (p' ,q')$, so that 
$ \deg (h^{\abrac{ 2 }})_{ (0)}  = ( p ,s)$ and 
$\deg (h^{\abrac{ 2 }})_{ (1)}  = (s , q)$ for some $s \in  \IZ^n$. 
By comparing the two sides of \eqref{p7}, we have
$ ( s , q) = (r , l)$ and $ (p' , q') = - ( p , s)$. Thus    
\eqref{eq:tau-bigrd} follows:
$q  = l$,   $p' = - p$ and $q' = -s = -r$.
\end{proof}

\begin{cor}
\label{cor:tau-bigrd} 
Let  $ h \in  H_{ (r , l)} $ be a homogeneous element with $\tau ( h )  = h^{\abrac{ 1 }} \ot_B h^{\abrac{ 2 }}$. Then, 
\begin{itemize}
\item[i)] for  $r = l $ one has $\deg h^{\abrac{1 }} + \deg h^{ \abrac{ 2 }}  = 0$; 
\item[]
\item[ii)] for $ r \neq l$,  one has $\tau ( h) = 0$. 
\end{itemize}
\end{cor}
\begin{proof}
With $r = l$, the first statement follows from \eqref{eq:tau-bigrd}. 
The latter also says that $ h^{\abrac{1 }} h^{\abrac{ 2 }} \in  A_{ (0 , l - r) }$,
which is non-zero unless $  h^{\abrac{1 }} \otimes  h^{\abrac{ 2 }} = 0$.
\end{proof}

This result is in accordance with \eqref{tauhexp} and \eqref{taukexp} by recalling that $
h_{jl}$ has bi-degree $((e_j, 0), (e_l, 0))$ while $k_{jl}$ has bi-degree $((e_j, 0), (-e_l, 0))$. 
It allows one to repeat the second part of Proposition \ref{prop:comodalg-theta}
and deform the starting Hopf--Galois extension into a new one.  

\begin{prop}\label{prop:H-G-Theta}
Consider the deformed pair $(H_\Theta , A_\Theta)$ obtained in {\rm Proposition 
\ref{prop:comodalg-Theta}} and define a deformed canonical Galois map $\chi_\Theta$ by
\begin{align}
\label{eq:chi-Theta}
\chi_\Theta: A_\Theta \otimes_{ B_\Theta }  A_\Theta
\to A_\Theta \otimes  H_\Theta , \quad 
a' \otimes_{ B_\Theta} a \mapsto  a' \cdot_\Theta a_{ (0)} \otimes a_{ (1)} .
\end{align}
This is invertible if and only if the starting canonical Galois map is with 
the same translation map, but viewed as a map $\tau : H_\Theta \to A_\Theta \otimes_{ B_\Theta }  A_\Theta$. 
\end{prop}
\begin{proof}

Let  $ h \in  H_{ (r , l)} $ be a homogeneous element with $\tau ( h )  = h^{\abrac{ 1 }} \ot_B h^{\abrac{ 2 }}$, 
the starting translation map.
From Corollary \ref{cor:tau-bigrd}, $\deg h^{\abrac{1 }} = - \deg h^{ \abrac{ 2 }}$. Then, with a slight abuse of notation 
($B = B_\theta$ as a vector space), from Lemma \ref{lem:mtheta-m},  
\begin{align*}
\chi_\Theta (\tuno{h} \ot_B \tdue{h}) & = \tuno{h} \cdot_\Theta \zero{\tdue{h}} \ot \one{\tdue{h}} = 
\tuno{h} \zero{\tdue{h}} \ot_B \one{\tdue{h}} \\
& = \chi(\tuno{h}\ot_B \tdue{h}) 
\end{align*}
and $\chi_\Theta$ is invertible if and only if $\chi$ is. Or, the pair $(H_\Theta , A_\Theta)$ 
is a Hopf--Galois extension if and only if the 
pair $(H, A)$ is such. 
\end{proof}

\section{Hopf algebroids} \label{se:had}
We are ready for the Hopf algebroid structure.
We start with a bialgebroid $  \C ( H_\Theta , A_\Theta)$ associated to the 
Hopf--Galois extension $(H_\Theta, A_\Theta)$ of the previous section. 
We next show that the flip can serve as an antipode. 
In \S\S ~\ref{sec:ev-q-shp} and \ref{se:alg-su2sym}, we present two examples. 
Firstly an algebroid for the principal $\SU(2)$-principal bundle 
over the four-sphere $\St^4$ decribed in \S\ref{sec:-q-shp} followed 
by the one for the bundles over the even spheres of \S\ref{sec:HG-SU2}.

\subsection{The bialgebroid $  \C ( H_\Theta , A_\Theta)$} \label{subsec:ES-alg} 
We know from Lemma \ref{lem:balancedts-deg} that the coinvariant elements for the action of 
$H_\Theta$ have trivial right grading. This will clearly be the case also for the coinvariant elements 
for the diagonal action that is needed for the Ehresmann--Schauenburg bialgebroid
Thus the construction of the bialgebroid will be the same for the Hopf--Galois extension $(H_\Theta, A_\Theta)$ 
in Proposition \ref{prop:H-G-Theta} of our scenario II and for the pair $( H , A_\theta)$ 
discussed for Scenario I in Proposition \ref{prop:comodalg-theta}. We describe the former here. 
 
Consider then the Hopf--Galois extension $(H_\Theta, A_\Theta)$. The diagonal coaction
is in \eqref{eq:coact-HTheta}:
\beq\label{eq:coact-Alg}
\delta^{ A \otimes_\Theta A}:
A_\Theta \otimes  A_\Theta \to  A_\Theta \otimes  A_\Theta \otimes  H_\Theta , 
\quad 
\delta^{ A \otimes_\Theta A} (a\ot \tilde a) = 
\zero{a} \otimes \zero{\tilde a} \otimes 
\one{a} \cdot_\Theta \one{\tilde a}.
\eeq
with $V = W = A$.
From the analysis before, and in particular from the fact that the canonical map and translation 
maps are the same as maps between vector spaces, the conclusion is that all the structure 
equations listed in
\S\ref{sec:ES-bialerd} hold true after deformation (which means 
replacing every occurrence of multiplication by the deformed one). 

\begin{lem}
\label{lem:ES-ald} 
Let $(H, A)$ be a Hopf--Galois extension that fulfil the assumptions 
on the bigradings of earlier and let $(H_\Theta , A_\Theta)$ be 
the deformed $(H_\Theta , A_\Theta)$ Hopf--Galois extension obtained in 
Proposition \ref{eq:chi-Theta}. 
Then the deformed coaction $\delta^{ A \otimes_\Theta A}$ in \eqref{eq:coact-Alg}, 
gives rise to the same coinvariant subspace 
as that of $\delta^{A \otimes  A}$:
\begin{align*}
\brac{ A_\Theta \otimes  A_\Theta }^{\mathrm{co} H_\Theta }  
= \brac{ A \otimes  A }^{\mathrm{co} H} . 
\end{align*}
Also, the deformed Ehresmann--Schauenburg bialgebroid
\begin{align*}
\C ( H_\Theta , A_\Theta) 
= \brac{ \brac{ A_\Theta \otimes  A_\Theta }^{\mathrm{co} H_\Theta } , 
\bullet_\Theta }
\end{align*}
with respect to $\C (A , H)$ in Def. \ref{def:reb}), has only the algebra structure changed, given by:
\begin{align}
\label{eq:ES-m-Theta} 
(x \otimes y) \bullet_\Theta ( \tilde x \otimes \tilde y) :=  
x \cdot_\theta \tilde x \otimes  \tilde y \cdot_\theta y \, .
\end{align}
\end{lem}
\begin{proof}
The results follows from the identification in Lemma \ref{lem-2vers}, which uses only the translation map that is unchanged (as a map between vector spaces) when deforming.
\end{proof}

\subsection{The flip map as the antipode} \label{subsec:flip-atpd} 
The bialgebroids of the previous section gets in fact a structure of Hopf algebroid with a suitable antipode.
Now, when the structure Hopf algebra $H$ is commutative, the flip map 
preserves the coinvariant elements of the diagonal coaction. Indeed,  given
\begin{align}
\label{eq:flip-S} 
S: A \otimes  A \to A \otimes  A , \quad  
a \otimes  \tilde  a \mapsto  
\tilde a \otimes a , 
\end{align}
for any  coinvariant $a \otimes  \tilde a \in  A \otimes  A$, 
by swapping $a$ and  $ \tilde a$ in 
$ a \otimes  \tilde a \otimes 1 =  
a_{ (0)}   \otimes \tilde  a_{ (0)} \otimes a_{ (1)} \tilde a_{ (1)} $,
we see that  $ \tilde a \otimes  a$ is  coinvariant as well:
$$
\tilde a \otimes  a \otimes  1 =
\tilde a_{ (0)}   \otimes  a_{ (0)} \otimes a_{ (1)} \tilde a_{ (1)}  
=
\tilde a_{ (0)}   \otimes  a_{ (0)} \otimes  \tilde a_{ (1)}  a_{ (1)} ,
$$
where the last equal sign invokes the commutativity of $H$. Therefore, 
when restricted to the  coinvariant subspaces the flip is a candidate for the antipode of 
$\C( H, A)$ and $\C ( H, A_\theta)$. 

In the more general situation, despite $H_\Theta$ needs no longer stay commutative after the $\theta$-deformation,
the flip $S$ still maps $\C ( H_\Theta , A_\Theta)$ into itself since we have shown in Lemma \ref{lem:ES-ald} 
that $\C ( H_\Theta , A_\Theta)$ and $ \C (H, A)$ are identical as vector spaces 
(This fact will be explicitly seen for the example in \S\ref{sec:ev-q-shp} below.) 

The main result of this section is that the flip $S$ makes 
$\C ( H_\Theta , A_\Theta)$ into a Hopf algebroid.
\begin{thm}
By only deforming multiplication related structures of the Hopf 
algebroid $\C( H, A)$ over $B$, the resulting $\C ( H_\Theta , A_\Theta)$ 
is a Hopf algebroid, but with base algebra $B_\Theta$.
\end{thm}
\begin{proof}
One needs to verify the compatibility conditions in 
\eqref{hopbroid1} and \eqref{hopbroid2}.  The latter one is the
less nontrivial one and is handled in  Lemma \ref{lem:anti-comp-I} below.
We point out that the computations below work for 
both $\C( H, A)$ and $\C ( H_\Theta , A_\Theta)$ since they do not rely
on the commutativity of the underlying algebra structures in the Hopf--Galois
extension.
\end{proof}
\begin{lem}
\label{lem:anti-comp-I} 
The flip $ S :  \C ( H_\Theta , A_\Theta) \to \C ( H_\Theta , A_\Theta)$ 
with $S^{-1} = S $ fulfils   
the compatibility conditions in \eqref{hopbroid2}, 
that is, for all $ h \in \C ( H_\Theta , A_\Theta)$:
\begin{align*}
(S^{-1} h_{ (2)} )_{ (1)}  \otimes_{B_\Theta}
(S^{-1} h_{ (2)} )_{ (2)} \bullet_\Theta   h_{ (1)} 
&=
S^{-1} h \otimes_{B_\Theta} 1 , 
\\
(S h_{ (1)} )_{ (1)} \bullet_\Theta  h_{ (2)} \otimes_{B_\Theta}
(S h_{ (1)} )_{ (2)} 
&=
1 \otimes_{B_\Theta} S(h) .
\end{align*}
\end{lem}
\begin{proof}
We shall prove the first one as an example and leave the second one to avid readers.

Write $h = a \otimes  \tilde a \in \C ( H_\Theta , A_\Theta)$,
where $a, \tilde a \in A_\Theta$, then
the coproduct in \eqref{copro} reads
\begin{align*}
\Delta (h) = h_{ (1)} \otimes_{B_\Theta}  h_{ (2)}   =
\brac{ a_{ (0)} \otimes  (a_{ (1)})^{\abrac{ 1 } }   }
\otimes_{B_\Theta} \brac{ (a_{ (1)})^{\abrac{ 2 } } \otimes  \tilde a } .
\end{align*}
We compute:
\begin{align*}
(S^{-1} h_{ (2)} )_{ (1)}\otimes_{B_\Theta}  & 
(S^{-1} h_{ (2)} )_{ (2)} \bullet_\Theta   h_{ (1)} 
\\ &= 
\brac{ \tilde a \otimes   (a_{ (1)})^{\abrac{ 2 } }  }_{ (1)} 
\otimes_{B_\Theta}  
\brac{ \tilde a \otimes   (a_{ (1)})^{\abrac{ 2 } }  }_{ (2)} 
\bullet_\Theta 
\brac{ a_{ (0)} \otimes  (a_{ (1)})^{\abrac{ 1 } }    }
\\ &=  
\tilde a_{ (0)} \otimes  (\tilde a_{ (1)})^{\abrac{ 1 }} 
\otimes_{B_\Theta}  
\brac{  ( \tilde a_{ (1)} )^{\abrac{ 2 }} \otimes (a_{ (1)})^{\abrac{ 2 } }  }
\bullet_\Theta 
\brac{  a_{ (0)} \otimes  (a_{ (1)})^{\abrac{ 1 } }  }
\\ &=
\tilde a_{ (0)} \otimes  (\tilde a_{ (1)})^{\abrac{ 1 }} 
\otimes_{B_\Theta}  
( \tilde a_{ (1)} )^{\abrac{ 2 }} \cdot_\Theta a_{ (0)}  \otimes 
(a_{ (1)})^{\abrac{ 1 } } \cdot_\Theta (a_{ (1)})^{\abrac{ 2 } }
\\ &= 
\tilde a_{ (0)} \otimes  (\tilde a_{ (1)})^{\abrac{ 1 }} 
\otimes_{B_\Theta}  
( \tilde a_{ (1)} )^{\abrac{ 2 }} \cdot_\Theta a_{ (0)}  \otimes 
\varepsilon ( a_{ (1)} ) 1_{ A_\Theta} 
\\ &=
\tilde a_{ (0)} \otimes  (\tilde a_{ (1)})^{\abrac{ 1 }} 
\otimes_{B_\Theta}  
( \tilde a_{ (1)} )^{\abrac{ 2 }} \cdot_\Theta a \otimes  1   ,
\end{align*}
where, in the last two steps,
we have used \eqref{p5} and  the compatibility between the counit 
$ \varepsilon: H_\Theta \to \mathbb{C}$ and the coaction $\delta^A$. 
To continue:
\begin{align*}
(S^{-1} h_{ (2)} )_{ (1)}  \otimes_{B_\Theta}
(S^{-1} h_{ (2)} )_{ (2)}  \bullet_\Theta  h_{ (1)} 
&= 
\brac{ 
\tilde a_{ (0)} \otimes  (\tilde a_{ (1)})^{\abrac{ 1 }} 
} \otimes_{B_\Theta}
\brac{ 
( \tilde a_{ (1)} )^{\abrac{ 2 }} \bullet_\Theta   a \otimes  1
} 
\\ & =   
\tilde a_{ (0)} \otimes \tau \brac{ \tilde a_{ (1)}  } \cdot_\Theta  a 
\\ & =
\tilde a \otimes  a \otimes_{ B_\Theta }  1 
= S^{-1} h \otimes_{ B_\Theta }   1 ,
\end{align*}
where we need \eqref{ec1}, which is an equivalent  description for
$a \otimes  \tilde a \in \C ( H_\Theta , A_\Theta)$,
to complete the second line.
\end{proof}

\subsection{The algebroid with $\SU(2)$-symmetry}\label{se:alg-su2sym}
With respect to the example in \S\ref{sec:HG-SU2},
denote $A=A(\Sk^7)$, $H=A(\SU(2))$ and $B=A(\St^4) = A^{co H}$ the subalgebra of invariants and, 
as usual $\delta^A(a) = \zero{a} \pot \one{a}$ and $\tau(h) = \tuno{h} \ot_B \tdue{h}$.

Consider then the diagonal coaction of $H$ on the tensor product algebra $A\ot  A$:
$$
\delta^{A\ot  A}: A\ot  A \to A\ot  A\ot  H, \quad a\ot  \tilde{a}  \mapsto
\zero{a}\ot  \zero{\tilde{a}} \ot   \one{a}\one{\tilde{a}} \, .
$$

\begin{lem}
The $B$-bimodule $\C(A,H)$ of coinvariant elements for the diagonal coaction is generated 
by elements of the tensor products $p \ot 1$ and $1\ot {q}$ together with 
$$
V = \Psi \pot \Psi^\dagger. 
$$
\begin{proof} 
It is clear that elements of $p \ot 1$ and $1\ot {q}$ are coinvariants. For $V = \Psi \pot \Psi^\dagger$:
$$
\delta^{A \ot  A}(V) = \zero{\Psi} \pot \zero{\Psi^\dagger} \pot \one{\Psi} \one{\Psi^\dagger}
= \Psi \ot \Psi^\dagger \pot (w w^\dagger)= V \pot (w w^\dagger) = V \pot \II_2 ,
$$
in parallel with the coinvariance \eqref{once}. 
\end{proof}
\end{lem}

 With the flip $\sigma(a \ot b) = b \ot a$ we define
\beq\label{antipode1case}
S_\C(V) := \sigma(\Psi \pot \Psi^\dagger) = V^\dagger, \quad \textup{or} \quad S_\C(V_{mn}) 
= V^\dagger_{mn} = \sum_r \Psi^\dagger_{rn} \pot \Psi_{mr}.
\eeq
Then, a direct computation shows that 
\begin{align} \label{alter}
S_\C(V) V &= V^\dagger V = 1 \ot \, \Psi \cdot_{op} \Psi^\dagger = 1 \ot \, {q} \nn \\
V S_\C(V)  &= V V^\dagger = \Psi \cdot \Psi^\dagger \, \ot 1 = p \, \ot 1 .
\end{align}
These then are relations among the elements of $V$ as generators of the 
$B$-bimodule $\C(A,H)$. The latter has the structure of a Hopf algebroid.

Firstly, 
the projections $V V^\dagger = p \ot 1$ and $V^\dagger V= 1\ot {q}$ are the two embedded copies of the 
4-sphere $A(\St^4)$ in $\C(A,H)$: $A(\St^4) \ot 1$ and $1 \ot A(\St^4)$, via source and target map respectively, 
as explicitly described in Lemma \ref{source-target} below. 

Next, 
according to the definition \eqref{copro} 
a coproduct $\Delta : \C(A,H) \to \C(A,H) \ot_B \C(A,H)$, is given on the matrix $V$ of generators by
\begin{align} \label{ccop}
\Delta (V) = \zero{\Psi} \pot \tuno{\one{\Psi}} \pot _B \tdue{\one{\Psi}} \pot \Psi^\dagger 
= \Psi \ot \Psi^\dagger \pot_B \Psi \ot \Psi^\dagger = V \pot _B V.
\end{align}
In components this reads:
\beq\label{copcom}
\Delta(V_{m n}) = \sum_r \, V_{m r} \ot_B V_{r n}. 
\eeq
From \eqref{ccop} one also gets
\beq\label{ccops}
\Delta(S_\C(V)) = \Delta(V^\dagger) = \sigma(V^\dagger \pot _B V^\dagger)
\eeq
or, in components:
\beq\label{copcomop}
\Delta(V^\dagger_{mr}) = \sum_r \, V^\dagger_{r n} \ot_B V^\dagger_{m r} .
\eeq

Finally, the map $S_\C$ in \eqref{antipode1case} is indeed an antipode for $\C(A,H)$. 
Since $S_\C$ is the flip, condition \eqref{hopbroid1} is obvious. 
We are left to show condition \eqref{hopbroid2}. For this, take $h=V$.
With expressions \eqref{ccop} and \eqref{ccops} for the coproducts, and using \eqref{alter}:  
\begin{align} \label{sverif}
\onet {(S_\C\one{h})} \two{h} \ot_B \twot {S_\C(\one{h})} 
& = \onet {(S(\one{V}))} \, \two{V} \pot_B \twot {(S(\one{V}))} \nn \\
& = \onet {(S_\C(V))} \, V \pot_B \twot {(S_\C(V))} \nn \\
& = \onet{(V^\dagger)} \, V \pot_B \twot{(V^\dagger)} 
= V^\dagger \, V \pot_B V^\dagger \nn \\
& = 1 \ot {q} \pot_B \, V^\dagger = 1 \ot 1 \ot_B \, {q} \, V^\dagger \nn \\
& = 1 \ot 1 \ot_B V^\dagger = 1 \ot 1 \ot_B S_\C(V).
\end{align}
Since elements of ${q}$ are in $B$ they can be crossed over $B$-tensor products, and we used the relation 
$q \,V^\dagger = V^\dagger$. 
The other condition in \eqref{hopbroid2} is similar since $S_\C^{-1} = S_\C$. 

\noindent
In components of $V$ this works as follows. 
Take $h=V_{mn}$, with $S_\C(h) = V^\dagger_{mn}$. Using the expressions for the coproduct, we compute:
\begin{align}\label{sverif-bis}
\onet {(S_\C\one{h})} \two{h} \ot_B \twot {S_\C(\one{h})} 
& = \sum_{r} \onet {(S_\C(V)_{m r})} V_{rn} \ot_B \twot {(S_\C(V)_{mr})} \nn \\
& = \sum_{r} \onet {(V^\dagger_{m r})} V_{rn} \ot_B \twot {(V^\dagger_{m r})} \nn \\
& = \sum_{rs} V^\dagger_{s r} V_{rn} \ot_B V^\dagger_{ms} = \sum_{s} (V^\dagger V)_{s n} \ot_B V^\dagger_{ms} \nn \\
& = \sum_{s} 1 \ot (\Psi \cdot_{op} \Psi^\dagger)_{s n}\ot_B V^\dagger_{ms} = \sum_{s} 1 \ot 1 \ot_B (\Psi \cdot_{op} \Psi^\dagger)_{s n} V^\dagger_{ms} 
 \nn \\
& = \sum_{s, j, k} 1 \ot 1 \ot_B \, \Psi^\dagger_{j n} \Psi_{s j} \Psi^\dagger_{ks} \ot \Psi_{mk} \nn \\
& = \sum_{s, j,k} 1 \ot 1 \ot_B \, \Psi^\dagger_{j n} (\Psi^\dagger_{ks}  \cdot_{op} \Psi_{s j} ) \ot \Psi_{mk} \nn \\
& = \sum_{j,k} 1 \ot 1 \ot_B \, \Psi^\dagger_{j n} ( \delta_{k j} 1 ) \ot \Psi_{mk} 
= \sum_{k} 1 \ot 1 \ot_B \, \Psi^\dagger_{k n} \ot \Psi_{mk} \nn \\
& = 1 \ot 1 \ot_B V^\dagger_{mn}  
= 1 \ot 1 \ot_B S_\C(V)_{mn} = 1 \ot 1 \ot_B S_\C(h) .
\end{align}
Here we crossed elements of $(\Psi \cdot_{op} \Psi^\dagger)$ 
over the $B$-tensor product,  since they are in $B$, and the relation 
$ \sum_{s} \Psi^\dagger_{ks}  \cdot_{op} \Psi_{s j} = \delta_{k j} 1$.

\subsubsection{Generators and relations}
In term of generators and relations let us write
\begin{align}\label{4matrices}
V = \begin{pmatrix}
P_1 & Q_2 \\
Q_1 & P_2
\end{pmatrix}  \qquad \textup{with} \qquad
P_1 &= \begin{pmatrix}
Z_0 & - \wt{X}_0 \\
X_0 & \wt{Z}_0
\end{pmatrix}, \quad 
Q_2 = \begin{pmatrix}
Z_2 & - \wt{W}_2 \\
W_2 & \wt{Z}_2
\end{pmatrix}, \nn \\
Q_1 & = \begin{pmatrix}
Z_1 & - \wt{W}_1 \\
W_1 & \wt{Z}_1
\end{pmatrix}, \quad
P_2 = \begin{pmatrix}
W_0 & - \wt{Y}_0 \\
Y_0 & \wt{W}_0
\end{pmatrix} . 
\end{align}
An explicit computation leads to
\begin{align} \label{subtens}
Z_0 & = \psi_1 \ot \psi^*_1 + \psi^*_2 \ot \psi_2 , \qquad \wt{Z}_0  = \psi^*_1 \ot \psi_1 + \psi_2 \ot \psi^*_2  
= Z^*_0 , \nn \\ 
X_0 , & = \psi_2 \ot \psi^*_1 - \psi^*_1 \ot \psi_2 , \qquad \wt{X}_0 = \psi^*_2 \ot \psi_1 - \psi_1 \ot \psi^*_2 
= X^*_0 \nn \\
W_0 & =  \psi_3 \ot \psi^*_3 + \psi^*_4 \ot \psi_4 , \qquad \wt{W}_0 = \psi^*_3 \ot \psi_3 + \psi_4 \ot \psi^*_4 
= W^*_0 , \nn \\ 
Y_0 & = \psi_4 \ot \psi^*_3 - \psi^*_3 \ot \psi_4 \qquad \wt{Y}_0 = \psi^*_4 \ot \psi_3 - \psi_3 \ot \psi^*_4 
= Y^*_0,  \nn \\ 
Z_1 & =  \psi_3 \ot \psi^*_1 + \psi_4^* \ot \psi_2 , \qquad \wt{Z}_1 = \psi^*_3 \ot \psi_1 + \psi_4 \ot \psi^*_2 
= Z^*_1,  \nn \\
W_1 & = \psi_4 \ot \psi^*_1 - \psi^*_3 \ot \psi_2 , \qquad \wt{W}_1 = \psi^*_4 \ot \psi_1 - \psi_3 \ot \psi^*_2 
= W^*_1,  \nn \\
Z_2 &=  \psi_1 \ot \psi^*_3 + \psi^*_2 \ot \psi_4 , \qquad \wt{Z}_2 = \psi_2 \ot \psi^*_4 + \psi^*_1 \ot \psi_3 
= Z^*_2,  \nn \\
W_2 & = \psi_2 \ot \psi^*_3 - \psi^*_1 \ot \psi_4 , \qquad \wt{W}_2 = \psi^*_2 \ot \psi_3 - \psi_1 \ot \psi^*_4 
= W^*_2. 
\end{align}
It is then immediate to check that 
$$
S_\C(V) = V^\dagger .
$$
We know that the generators are not independent. Indeed:
\begin{lem} There are four sphere relations:
\begin{align}
\wt{Z}_0 Z_0 + \wt{X}_0 X_0 &= Z_0 \wt{Z}_0 + X_0 \wt{X}_0 = \zeta_0 \ot \zeta_0 , \nn \\
\wt{W}_0 W_0 + \wt{Y}_0 Y_0  &= W_0 \wt{W}_0 + Y_0 \wt{Y}_0 = (1-\zeta_0) \ot (1-\zeta_0) , \nn \\
\wt{Z}_1 Z_1 + \wt{W}_1 W_1 &= Z_1 \wt{Z}_1 + W_1 \wt{W}_1 = (1-\zeta_0) \ot \zeta_0 , \nn \\
\wt{Z}_2 Z_2 + \wt{W}_2 W_2 &= Z_1 \wt{Z}_1 + W_1 \wt{W}_1 = \zeta_0 \ot (1-\zeta_0) .
\end{align}
\begin{proof}
One computes these from the relations \eqref{alter}.
\end{proof} 
\end{lem}
In a sense this says that the four matrices in \eqref{4matrices} are all equivalent and for the generators of the $B$-bimodule 
$\C(A,H)$ of coinvariant elements one can take any one of those together with $A(\St^4) \ot 1$ and $1 \ot A(\St^4)$.
Alternatively, one could express the generators of the latter spheres in terms of the generators in \eqref{subtens}.
\begin{lem}\label{source-target}
The source map:
\begin{align}
& \wt{Z}_0 Z_0 + \wt{X}_0 X_0 + \wt{Z}_2 Z_2 + \wt{W}_2 W_2 
=\zeta_0 \ot 1 , \nn \\
& Z_0 \wt{Z}_1 + \wt{X}_0 W_2 + Z_2 \wt{W}_0 + \wt{W}_2 Y_0 
= \zeta_1 \ot 1 , \nn \\
& X_0 \wt{Z}_1 + W_2 \wt{W}_0 - \wt{Z}_0 W_1 - \wt{Z}_2 Y_0 
= \zeta_2 \ot 1 .
\end{align}
and the target map:
\begin{align}
& \wt{Z}_0 Z_0 + \wt{X}_0 X_0 + \wt{Z}_1 Z_1 + \wt{W}_1 W_1 = 1 \ot \zeta_0, \nn \\
& W_2 \wt{X}_0 + Z_2 \wt{Z}_0 + Y_0 \wt{W}_1 + W_0 \wt{Z}_1 = 1\ot \zeta_1, \nn \\
& W_2 Z_0 - Z_2 X_0 + Y_0 Z_1 - W_0 W_1 = 1 \ot \zeta_2 .
\end{align}

\begin{proof}
The direct way for these is just to use again the relations in \eqref{alter}.
\end{proof} 
\end{lem}

\subsection{A Hopf algebroid with quantum orthogonal symmetry} \label{sec:ev-q-shp} 
Let us denote $A=\O(\SOt(2n+1))$, $H=\O(\SOt(2n))$ and $B=\O(S_\theta^{2n})$. 
With the notations of \S\ref{sec:-q-shp} consider the matrix valued function
\beq
\Phi = (\Phi_{J K}) =
\begin{pmatrix} 
\b{a} & \b{b} \\ 
\b{b}^* & \b{a}^* \\
\b{v} & \b{v}^* 
\end{pmatrix} .
\eeq

Then, the orthogonality conditions $N^\dagger N = \II$ gives that $\Phi^\dagger \cdot \Phi = \II_{2n}$.
Moreover, the entries of the matrix $\Phi \cdot \Phi^\dagger$ 
(a projection from the condition  $\Phi^\dagger \cdot \Phi = \II_{2n}$) are coinvariants 
for the coaction \eqref{subha}.
In fact, one computes explicitly that
\beq\Phi \cdot \Phi^\dagger = 
\begin{pmatrix}
1 - \b{u}  \b{u}^\dagger & - \b{u} \b{u}^t & - \b{u} x \\ 
 - \b{u}^* \b{u}^\dagger & 1 - \b{u}^* \b{u}^t & - \b{u}^* x \\
- x \b{u}^\dagger & - x \b{u}^t & 1 - x^2 
\end{pmatrix}
\eeq
Let us denote
\beq
\b{w}=\begin{pmatrix}
{\bf h} & {\bf k}  \\
{\bf k}^* & {\bf h}^* 
\end{pmatrix}, 
\eeq
the defining matrix of $\O(\SOt(2n))$, with $\b{w}^\dagger \b{w} = \b{w} \b{w}^\dagger = \II_2$. 
Then the coaction \eqref{subha} reduces to a coaction
\beq\label{cored}
\delta^A (\Phi) = \Phi \pot \b{w} 
\eeq
or 
\beq\label{cored-bis}
\begin{matrix}
\delta^A (\b{a}) = \b{a} \pot \b{h} + \b{b} \pot \b{k}^* \\
\delta^A (\b{b}) = \b{a} \pot \b{k} + \b{b} \pot \b{h}^* \\
\delta^A (\b{v}) = \b{v} \pot \b{h} + \b{v}^* \pot \b{k}^* 
\end{matrix} \qquad \mbox{and} \qquad 
\begin{matrix}
\delta^A (\b{a}^*) = \b{a}^* \pot \b{h}^* + \b{b}^* \pot \b{k} \\
\delta^A (\b{b}^*) = \b{a}^* \pot \b{k}^* + \b{b}^* \pot \b{h} \\
\delta^A (\b{v}^*) = \b{v}^* \pot \b{h}^* + \b{v} \pot \b{k} \, . 
\end{matrix}
\eeq
In turns this gives 
\beq
\delta^A (\Phi^\dagger) = \sigma(\b{w}^\dagger \pot \Phi^\dagger).
\eeq
The translation map is easily seen to be given by
\beq \label{tramap2}
\tau(\b{w}) = \Phi^\dagger \pot_B \Phi .   
\eeq

One could show the coinvariance of the entries of $\Phi \cdot \Phi^\dagger$ by using the explicit form \eqref{cored-bis} of the coaction. 
Then in exactly the same way, one shows the following: 
\begin{lem}
The $B$-bimodule $\C(A,H)$ of coinvariant elements for the diagonal coaction of $H$ on $A \ot A$ is generated 
by elements $1\ot \b{u}, 1\ot \b{u}^*, 1 \ot x$, and $\b{u} \ot 1, \b{u}^* \ot 1, x \ot 1$, 
together with the entries of the matrix
$$ 
\b{V} = \Phi \pot \Phi^\dagger =
\begin{pmatrix} 
\b{a} & \b{b} \\ 
\b{b}^* & \b{a}^* \\
\b{v} & \b{v}^* 
\end{pmatrix} \pot 
\begin{pmatrix} 
\b{a} & \b{b} \\ 
\b{b}^* & \b{a}^* \\
\b{v} & \b{v}^* 
\end{pmatrix}^\dagger.
$$
Moreover, the flip $S_\C(x \ot y) = y \ot x$ leaves unchanged the space of coinvariants, and in particular
$$
S_\C (\b{V}) = \sigma(\Phi \pot \Phi^\dagger) = \b{V}^\dagger .
$$
\end{lem}
\begin{proof}
For $\b{V} = \Phi \pot \Phi^\dagger$:
$$
\delta^{A \ot  A}(\b{V}) = \zero{\Phi} \pot \zero{\Phi^\dagger} \pot \one{\Phi} \one{\Phi^\dagger}
= \Phi \pot \Phi^\dagger \pot (\b{w} \b{w}^\dagger)= \b{V} \pot (\b{w} \b{w}^\dagger) = \b{V} \pot \II_{2n} 
$$
and directly: 
$\delta^{A \ot  A}(\b{V}^\dagger) = \sigma(\zero{\Phi} \pot \zero{\Phi^\dagger}) \pot \one{\Phi^\dagger} \one{\Phi}
= \sigma(\Phi \pot \Phi^\dagger) \pot (\b{w}^\dagger \b{w}) = \b{V^\dagger} \pot \II_{2n}
$. 
\end{proof}

\begin{lem} Using the conditions $N^\dagger N = \II = N N^\dagger$ one finds the following relations
\begin{align}
\Phi^\dagger \cdot_{op} \Phi = 
\begin{pmatrix}
(\Phi^\dagger \cdot \Phi )_{22} & (\Phi^\dagger \cdot \Phi )_{12} \\ 
(\Phi^\dagger \cdot \Phi )_{21} & (\Phi^\dagger \cdot \Phi )_{11} 
\end{pmatrix} =
Q (\Phi^\dagger \cdot \Phi)^t Q = 
\begin{pmatrix}
\II_n & 0 \\ 
0 & \II_n 
\end{pmatrix} = \II_{2n} 
\end{align}
\begin{align}
\Phi \cdot_{op} \Phi^\dagger =
\begin{pmatrix}
1 - \b{u}^* \b{u}^t & - \b{u} \b{u}^t & - \b{u}^* x\\ 
 - \b{u}^* \b{u}^\dagger & 1 - \b{u} \b{u}^\dagger & - \b{u} x \\
- x \b{u}^t & - x \b{u}^\dagger & 1 - x^2 
\end{pmatrix} =
Q (\Phi \cdot \Phi^\dagger)^t Q
\end{align}
It is evident that $(\Phi \cdot \Phi^\dagger)_{J K} \in B$; as well as $(\Phi \cdot_{op} \Phi^\dagger)_{J K} \in B$. 
Also,
\beq
\b{V} \b{V}^\dagger =  (\Phi \cdot \Phi^\dagger) \ot 1 , \qquad \b{V}^\dagger \b{V}  = 1 \ot (\Phi \cdot_{op} \Phi^\dagger),
\eeq
which express relations among the generators of $\C(A,H)$.
\end{lem}
\noindent
In parallel with the projection $\Phi \cdot \Phi^\dagger$, the matrix $\Phi \cdot_{op} \Phi^\dagger$ is a projection due to $\Phi^\dagger \cdot_{op} \Phi = \II_{2n}$.

\medskip
We are ready for the Hopf algebroid structure.
\begin{prop}
On $\C(A,H)$, the coproduct $\Delta : \C(A,H) \to \C(A,H) \ot_B \C(A,H)$, according to the definition \eqref{copro} and using the translation map \eqref{tramap2}, 
is given by
\begin{align} \label{ccopor}
\Delta (\b{V}) = \zero{\Phi} \pot \tuno{\one{\Phi}} \pot _B \tdue{\one{\Phi}} \pot \Phi^\dagger 
= \Phi \pot \Phi^\dagger \pot_B \Phi \pot \Phi^\dagger = \b{V} \pot _B \b{V}. 
\end{align}
In components this reads:
\beq\label{copcom-1}
\Delta(V_{J K}) = \sum_L V_{J L} \ot_B V_{L K}. 
\eeq
Also,
\beq
\Delta (S_\C(\b{V})) = \sigma(S_\C(\b{V}) \pot _B S_\C(\b{V})) , \quad 
\eeq
or 
\beq
\Delta(S_\C(V)_{J K}) = \sum_L S_\C(V)_{L K} \ot_B S_\C(V)_{J L} , \qquad \Delta(V^\dagger_{J K}) = \sum_L V^\dagger_{L K} \ot_B V^\dagger_{J L}.
\eeq
\end{prop}

Very much in the lines of the proof \eqref{sverif-bis}, we have the following:
\begin{prop}
The flip $S_\C$ is the antipode of $\C(A,H)$. 
\end{prop}
\begin{proof}
Since $S_\C$ is just the flip, condition \eqref{hopbroid1} is obvious. For conditions \eqref{hopbroid2} take $h=V_{J K}$, with $S_\C(h) = V^\dagger_{J K}$ and use the explicit form of the coproduct \eqref{copcom-1}. Then the proof proceeds verbatim as in the proof of \eqref{sverif-bis}.
\end{proof}

\appendix
\section{Deforming compact Lie groups along a toral subgroup} \label{sec:Qgrp} 

We recall Rieffel's construction in \cite{Rieffel:1993tw} that deforms 
function algebras of compact Lie groups and make comparison with the algebraic 
setup in \S\ref{subsec:ScII}. 
Let $G$ be a compact Lie group and  $K  \subset G$ be a toral subgroup of rank $n$ 
(for example, but not necessarily, a maximal torus).  Denote by $H = \mathcal O(G)$ the Hopf
algebra  of representative functions on $G$. One would like to only vary the
multiplication of $H$ to get a new Hopf algebra  $H_\Theta$. 
In order to retain the compatibility between the algebra and coalgebra
structures, one must 
cautiously pick the torus action and keep track of the equivariant
properties of the structure maps of $H$. 

A suitable deformation begins with an action $\alpha$ of
$T = K \times K$  on $G$ from two sides:
\begin{align}
\label{eq:alp-Tact} 
\alpha_{ k_1 , k_2 } ( g) = k_1^{-1} g k_2 , \quad 
k_1 , k_2 \in  K , \, \,  g \in  G
\end{align}
which gives rise to a $\IZ^n$-bigrading on $H$ via Pontryagin duality,
and a matrix of deformation like in \eqref{eq:Theta-matrix}:
\beq
\Theta = 
\begin{bmatrix}  \theta   & 0 \\   0 & - \theta   \end{bmatrix}  ,
\eeq
with $\theta$ a $n\times n$ antisymmetric matrix.

Of course, the pointwise multiplication between functions is indeed  
equivariant and thus respects the bigrading as in \ref{eq:prd-bigrd}.
Potential problems appear with the observation that
after deformation, on the one hand the algebra structure of
$ \mathcal O(G)_\Theta \otimes \mathcal O(G)_\Theta $ 
is inherited from
$\mathcal O( G \times  G)_{ \Theta \oplus \Theta}^{\alpha \otimes  \alpha}$
whose underlying torus action is
$\alpha \otimes  \alpha $ of $T \times  T$ on 
$G \times  G$:  in more detail, 
$  \alpha \otimes  \alpha :
\mathcal O ( G \times G) \to \mathcal O ( G \times G) $ is given by 
\begin{align}
\label{eq:alp-ox-alp-Tact} 
\alpha_{ k_1 , k_2 }   \otimes \alpha_{ k_1' , k_2'} 
(f) ( g_1, g_2) = 
f\brac{ k_1^{-1} g_1 k_2 , (k_1')^{-1} g_2 k_2' } ,
\end{align}
with $f \in  \mathcal O ( G \times  G)$ and 
$k_1 , k_2, k_1' , k_2' \in K $ and $g_1 , g_2 \in  G$. On the other hand, 
for this action the coproduct 
\begin{align*}
\Delta: \mathcal O (G) \to \mathcal O(G \times G) , \quad  
\Delta (f) ( g_1 , g_2 ) = f ( g_1 g_2) , \, \, 
\end{align*}
where $ f \in  \mathcal O(G)$ and  $ g_1 , g_2 \in  G$, 
is not equivariant. 
Nevertheless, the image of $\Delta$ is contained in the subalgebra 
\begin{align}
\label{eq:subspace-C}
\mathcal D = \big\{ 
f \in \mathcal O ( G \times  G) : 
f( g_1 k , g_2) = f ( g_1 , k g_2) , \, \, 
g_1 , g_2 \in  G, \, \,  k \in  K
\big\} ,
\end{align}
and there is a  $T  = K \times K$ action $\beta$ on $ \mathcal D$, given by:
\begin{align}
\label{eq:beta-action} 
\beta_{ k_1 , k_2}  (f) (g_1 , g_2 ) = f \brac{ k_1^{-1} g_1 , g_2 k_2 },
\end{align}
such that:
\begin{itemize}
\item[i)]
the coproduct $\Delta : \mathcal O (G) \to \mathcal D$ is equivariant ,
\item[] ~
\item[ii)] 
the $\theta$-deformation $ \mathcal D^\beta_{ \Theta} $ is a subalgebra of
$ \mathcal O( G \times  G)_{ \Theta \oplus \Theta}^{\alpha \otimes  \alpha}$ .
\end{itemize}
From the bigrading point of view, the subspace $ \mathcal D$
corresponds to 
\begin{align*}
\bigoplus_{r, s ,l \in \IZ^n } H_{ (r, s)} \otimes H_{ (s, l)} \subset H \otimes  H,
\end{align*}
and $\Delta$ being equivariant is exactly the dual condition of
\eqref{eq:coprd-bigrd}.

The counit $\varepsilon : \mathcal O ( G) \to \mathbb{C}$ is not equivalent
in any way, but it factors through
\begin{align*}
\varepsilon :  \mathcal O ( G ) \xrightarrow{\pi} \mathcal O ( K )
\xrightarrow{ \varepsilon_{ \mathcal O (K)} } \mathbb{C} ,
\end{align*}
where the projection $\pi (f) = f |_K $ for $f \in  \mathcal O ( G)$ 
is the restriction map of functions on $G$ onto the subgroup  $K$,
and $\pi$ is  equivariant when $K$ is equipped  with  
the action $\alpha$.  Hence, it is still an equivariant 
algebra homomorphism viewed as
$\pi: \mathcal O (G)_\Theta \to \mathcal O ( K)_\Theta$. 
Also, since $K$ is abelian with the choice of $\Theta$ in \eqref{eq:Theta-matrix},
it is not difficult to see that such $\theta$-deformation alters nothing:
$ \mathcal O ( K)_\Theta = \mathcal O ( K)$.
All these properties are reflected in term of the bigrading in the condition 
\eqref{eq:counit-bigrd}.
For the antipode and the $*$-operator 
\begin{align*}
S (f ) ( g)  = \overline{f(g^{-1})}, \quad
f^*(g) = \overline{f (g)} , \qquad  \forall  f \in  \mathcal O (G)
\end{align*}
it is a routine verification to check that for the bigradings, \eqref{eq:antp-star-gd} is indeed satisfied. 
 
To sum up, let $K$ be a torus subgroup of  $G$ of rank  $n$ and  $G$ is
a subgroup  of $ \widetilde G$. 
Denote by $H = \mathcal O ( G)$ and 
$ \widetilde H = \mathcal O  (\widetilde G)$ the associated Hopf algebras of
representative functions. They can be deformed to two new Hopf algebras 
$H_\Theta$ and $ \widetilde H_\Theta$, thanks to Proposition \ref{prop:H-Theta}. 

If one forgets the coalgebra structure and view $ \widetilde H$
as a $H$-comodule algebra, one has to repeat some of the 
arguments of earlier to check that the coaction 
\begin{align*}
\delta^{ \widetilde H} : \widetilde H \to \widetilde H \ot H , \quad  
\delta^{ \widetilde H} (f ) ( \tilde g, g) = f ( \tilde g , g), \, \quad 
\tilde g \in  \widetilde G, g \in  G
\end{align*}
indeed satisfies all the requirements of Proposition \ref{prop:comodalg-Theta}.
As a result, one obtaines a quantum homogeneous space given by 
$(H_\Theta , \widetilde H_\Theta)$, with coinvariant subalgebra 
$B_\Theta = \widetilde H_\theta^{\mathrm{co} H_\Theta }$ which plays the role of 
coordinate functions on the base. 
Thus, in this way, one $\theta$-deforms the Hopf algebra structures (on  $H$
and  $ \widetilde H$)  and the $H$-comodule algebra structure of $\widetilde H$
in such a way that compatibilities among the three survive.

\end{document}